\documentclass[12pt,reqno,a4]{amsart}
\usepackage{a4wide}

\usepackage{amsmath,amscd}
\usepackage{amsmath}
\usepackage{amsfonts}
\usepackage{amsthm}
\usepackage{amssymb}
\usepackage{enumitem}
\usepackage{mathrsfs}
\usepackage[x11names, svgnames,dvipsnames, rgb,table]{xcolor}

\theoremstyle{plain}
\newtheorem{Thm}{Theorem}[section]
\newtheorem{Cor}[Thm]{Corollary}
\newtheorem{Lem}[Thm]{Lemma}
\newtheorem{Prop}[Thm]{Proposition}
\theoremstyle{definition}		
\newtheorem{Def}[Thm]{Definition}
\newtheorem{Rem}[Thm]{Remark}
\newtheorem{Ex}[Thm]{Example}

\newtheorem{con}[Thm]{Construction}

\newcommand{\ar}{{\mathcal{R}}}
\newcommand{\el}{{\mathcal{L}}}
\newcommand{\eh}{{\mathcal{H}}}
\newcommand{\dee}{{\mathcal{D}}}
\newcommand{\jay}{{\mathcal{J}}}
\newcommand{\leqel}{{\leq_{\mathcal{L}}}}
\newcommand{\leqj}{{\leq_{\mathcal{J}}}}
\newcommand{\geqj}{{\geq_{\mathcal{J}}}}
\renewcommand{\l}{\mbox{${\langle}$}}
\renewcommand{\r}{\mbox{${\rangle}$}}

\renewcommand{\phi}{\varphi}

\newcommand{\N}{\mathbb{N}}

\DeclareMathOperator{\Inj_X}{\mathcal{I}\hspace{-0.1em}\mathit{nj}_X}

\newenvironment{thmenumerate}{\begin{enumerate}[label=\textup{(\arabic*)},leftmargin=10mm]}{\end{enumerate}}

\begin{document}
\title{On minimal ideals in pseudo-finite semigroups}
\author{Victoria Gould, Craig Miller, Thomas Quinn-Gregson, Nik Ru{\v s}kuc}
\address{Department of Mathematics, University of York, UK, YO10 5DD}
\email{victoria.gould@york.ac.uk, craig.miller@york.ac.uk}
\address{~\vspace{-1em}}
\email{tquinngregson@gmail.com}
\address{School of Mathematics and Statistics, St Andrews, Scotland, UK, KY16 9SS}
\email{nik.ruskuc@st-andrews.ac.uk}
\maketitle
\begin{abstract} 
A semigroup $S$ is said to be right pseudo-finite if the universal right congruence can be generated by a finite set $U\subseteq S\times S$, and there is a bound on the length of derivations for an arbitrary pair $(s,t)\in S\times S$ as a consequence of those in $U$.
This article explores the existence and nature of a minimal ideal in a right pseudo-finite semigroup.  
Continuing the theme started in an earlier work by Dandan et al., we show that in several natural classes of monoids, right pseudo-finiteness implies the existence of a completely simple minimal ideal.  This is the case for orthodox monoids, completely regular monoids and right reversible monoids, which include all commutative monoids.
We also show that certain other conditions imply the existence of a minimal ideal, which need not be completely simple; notably, this is the case for semigroups in which one of the Green's pre-orders $\leqel$ or $\leqj$ is left compatible with multiplication.
Finally, we establish a number of examples of pseudo-finite monoids without a minimal ideal.  We develop an explicit construction that yields such examples with additional desired properties, for instance, regularity or $\jay$-triviality.
\end{abstract}
~

\section{Introduction}
\label{sec:intro}

The notion of being pseudo-finite for semigroups arises from a variety of sources and may be expressed in several different ways, as explained in \cite{Dandan:2019}.  The simplest way of approaching this condition is via the universal relation, regarded as a one-sided congruence.  Informally, a semigroup $S$ is right (resp.\ left) pseudo-finite if the universal right (resp.\ left) congruence on $S$ is finitely generated and there is a bound on the length of sequences required to relate any two elements.  A more precise definition will be given in Section~\ref{sec:pfintro}.\par
The property of being (left) pseudo-finite was introduced by White in \cite{White:2017} in the language of ancestry.  This work was motivated by a conjecture of Dales and {\. Z}elazko, which states that a unital Banach algebra in which every maximal left ideal is finitely generated is necessarily finite dimensional.  One of the main results of \cite{White:2017} states that a monoid $M$ is left pseudo-finite if and only if the augmentation ideal $l_1^0(M)$ is finitely generated \cite[Theorem 1.7]{White:2017}.

In \cite{Dandan:2019}, Dandan et al.\ systematically studied the condition of being left pseudo-finite, 
within the broader context of semigroups having finitely generated universal left congruence.
These two conditions are certainly finiteness conditions (that is, every finite semigroup satisfies them).  
The latter condition was shown to be equivalent to several other concepts, which have previously been studied in different areas, e.g.\ the homological finiteness property of being type left-$FP_1$ \cite[Theorem 3.10]{Dandan:2019}.  For a group $G,$ the universal left (or right) congruence on $G$ is finitely generated if and only if $G$ is  a finitely generated group, and $G$ is left (or right) pseudo-finite if and only if it is finite \cite[Proposition 2.7]{Dandan:2019}.  In fact, it was noted in \cite{White:2017} that for weakly right cancellative monoids, which include groups, being left pseudo-finite coincides with being finite.
This is far from true for arbitrary semigroups and monoids, as will become apparent.  For example, any monoid with zero is left (and right) pseudo-finite.  Of course, a zero is precisely a trivial minimal ideal.

If a semigroup contains no proper ideals then it is said to be \emph{simple}.
A \emph{minimal} (left/right/two-sided) \emph{ideal} of a semigroup is a (left/right/two-sided) ideal containing no proper (left/right/two-sided) ideals.
If a semigroup $S$ has a minimal two-sided ideal, it is unique and is a simple subsemigroup.  
If a simple semigroup contains minimal left and right ideals it is said to be \emph{completely simple}.
One strand of \cite{Dandan:2019} concerns the existence and nature of a minimal ideal in a (left) pseudo-finite semigroup.
This was partly motivated by a question posed to Gould by Dales and White, asking whether every pseudo-finite semigroup is isomorphic to a direct product of a semigroup with zero by a finite semigroup.  This question was answered negatively in \cite[Example 7.7]{Dandan:2019}.  
On the other hand, it was shown in \cite{Dandan:2019} that every pseudo-finite semigroup that is inverse or a union of groups necessarily contains a completely simple minimal ideal.  It was noted in \cite[Remark 8.9]{Dandan:2019} that \cite[Theorem 8.1]{Dandan:2019} (which concerned the universal left congruence being finitely generated) could potentially be adapted to provide necessary and sufficient conditions for a semigroup with a completely simple minimal ideal to be pseudo-finite.  It was then observed in \cite[Open Question 8.10]{Dandan:2019} that every pseudo-finite semigroup hitherto considered possesses a completely simple minimal ideal, and the authors raised the question as to whether all pseudo-finite semigroups have this property.  (A positive answer would then yield a complete description of all pseudo-finite semigroups.)  This problem was later solved in the negative: in an article investigating the related condition that {\em every right congruence of finite index is finitely generated}, a counterexample was provided by way of a pseudo-finite simple (but not completely simple) semigroup \cite[Remark 7.3]{Miller:2020}.

The above progress still leaves open the possibility that every pseudo-finite semigroup possesses a {\em minimal ideal}.  The aim of the present paper is to systematically explore the existence and nature of a minimal ideal in a (right) pseudo-finite semigroup.
After some generalities concerning semigroups and their actions in Section \ref{sec:prelim}, the notion of pseudo-finiteness is introduced in Section \ref{sec:pfintro}.
The main theme of the paper, i.e.\ the relationship between pseudo-finiteness and minimal ideals, is properly started in Section \ref{sec:pf,ideals}.
The next four sections contain the main results of the article.
In Sections \ref{sec:csmi} and \ref{sec:mi} we exhibit a number of natural classes of semigroups within which pseudo-finiteness implies the existence of a minimal ideal, and, often, a completely simple minimal ideal.
It  turns out, however, that in general pseudo-finiteness need not imply the existence of a minimal ideal.
We present a specific transformation semigroup that is pseudo-finite but has no minimal ideal at the beginning of Section \ref{sec:nomids}.
To enable us to provide further such examples, in Section~\ref{sec:con} we introduce a general construction based on an ideal extension of a Rees matrix semigroup.  This construction is then deployed in the remainder of Section \ref{sec:nomids} to exhibit pseudo-finite monoids without a minimal ideal that possess some additional desirable properties.  The article concludes with some open questions and directions for future research in Section~\ref{sec:conclusion}.

%

\section{Preliminaries: Semigroups, Ideals and Actions}
\label{sec:prelim}

In this section we establish some basic definitions and facts about semigroups and actions.  We refer the reader to \cite{Howie:1995} for a more comprehensive introduction to semigroup theory, and to \cite{kkm:2000} for further details on actions.

Unless stated otherwise, $S$ will always denote a semigroup  
and $S^1$ the monoid obtained from $S$ by adjoining an identity (if $S$ is already a monoid, then $S^1$ has a \emph{new} identity).  
We denote the set of idempotents of $S$ by $E(S).$  If $S=E(S),$ it is called a {\em band}.  A {\em semilattice} is a commutative band. 

An element $a\in S$ is said to be {\em regular} if there exists $b\in S$ such that $a=aba.$  The semigroup $S$ is said to be {\em regular} if every element of $S$ is regular.  It turns that for every regular element $a\in S$ there exists $b\in S$ such that $a=aba$ and $b=bab$; in this case, the element $b$ is said to be an {\em inverse} of $a,$ and vice versa.  If $S$ is regular and each of its elements has a unique inverse, then $S$ is called {\em inverse}.  If $S$ is inverse, then its set of idempotents $E(S)$ forms a semilattice.

A non-empty subset $I\subseteq S$ is said to be a {\em right ideal} of $S$ if $IS\subseteq I.$  Left ideals are defined dually, and an {\em ideal} of $S$ is a subset that it is both a right ideal and a left ideal.
A right ideal $I$ of $S$ is said to be {\em generated by} $X\subseteq I$ if $I=XS^1.$  A right ideal $I$ is said to {\em finitely generated} if it can be generated by a finite set, and $I$ is said to be {\em principal} if it can be generated by a one-element set.

A {\em right congruence} on $S$ is an equivalence relation $\rho$ on $S$ such that $(a, b)\in\rho$ implies $(ac, bc)\in\rho$ for all $a, b, c\in S$; {\em left congruences} are defined analogously. The importance of one-sided congruences for monoids is that they determine monogenic (single-generated) actions; one-sided ideals are not sufficient for this.

Green's relations $\el$, $\ar$, $\eh$, $\dee$ and $\jay$ are standard tools for describing the ideal structure of a semigroup.
Green's preorder $\leqel$ on $S$ is given by
$$a\,\leqel\,b\Leftrightarrow S^1a\subseteq S^1b,$$
and this leads to the $\el$-relation:
$a\,\el\,b$ if and only if $a\,\leqel\,b\text{ and }b\,\leqel\,a.$
In other words, $a, b\in S$ are $\el$-{\em related} if and only if they generate the same principal left ideal.
The preorders $\leq_{\ar}$ and $\leq_{\jay}$ are defined analogously and yield the equivalences $\ar$ and $\jay,$ respectively.
Next we have $\eh=\ar\cap\el,$ and finally $\dee=\ar\circ\el(=\el\circ \ar=\el\vee\ar)$.
It is clear from the definitions that Green's relations are equivalences on $S.$ 
Moreover, the preorder $\leqel$ is right compatible and hence $\el$ is a right congruence, and similarly $\ar$ is a left congruence on $S.$ 
It is easy to see that the following inclusions between Green's relations hold: 
$$\mathcal{H}\subseteq\mathcal{L},\, \mathcal{H}\subseteq\mathcal{R},\, \mathcal{L}\subseteq\mathcal{D},\, \mathcal{R}\subseteq\mathcal{D},\, \mathcal{D}\subseteq\mathcal{J}.$$
Following the standard convention, we will denote the $\mathcal{L}$/$\mathcal{R}$/$\mathcal{H}$/$\mathcal{D}$/$\mathcal{J}$-class of $a\in S$ by $L_a$/$R_a$/$H_a$/$D_a$/$J_a.$ 

It can be easily shown that every right/left/two-sided ideal is a union of $\ar/\el/\jay$-classes.
A semigroup with no proper right/left ideals is called {\em right/left simple}.  A semigroup is called {\em simple} if it has no proper ideals.  Clearly if $S$ is right or left simple, then it is simple.

A right/left/two-sided ideal $I$ of $S$ is said to be {\em minimal} if there is no right/left/two-sided ideal of $S$ properly contained in $I$.   
It turns out that, considered as semigroups, minimal right/left ideals are right/left simple \cite[Theorem 2.4]{Clifford:1948}, and minimal ideals are simple \cite[Theorem 1.1]{Clifford:1948}.  The semigroup $S$ contains at most one minimal ideal, but $S$ may possess multiple minimal right/left ideals.  If $S$ has a minimal right/left ideal, then the minimal ideal exists and is equal to the union of all the minimal right/left ideals \cite[Theorem 2.1]{Clifford:1948}.

A {\em completely simple} semigroup is a simple semigroup that possesses both minimal right ideals and minimal left ideals.  A semigroup has both minimal right ideals and minimal left ideals if and only if it has a completely simple minimal ideal \cite[Theorem 3.2]{Clifford:1948}.
In particular, every finite semigroup has a completely simple minimal ideal.

Suppose that $S$ is a semigroup with a completely simple minimal ideal $K.$
Then the minimal right ideals of $K$ are also the minimal right ideals of $S$; let us denote them by $R_i$ ($i\in I$).
Similarly, let $L_j$ ($j\in J$) be the minimal left ideals of $K$ (and hence of $S$).
The intersections $H_{ij}=R_i\cap L_j$ are $\eh$-classes of $S$ and are isomorphic groups.
For $h\in H_{ij}$ and $s\in S,$ we have $hs\in R_i$ and $sh\in L_j,$ due to the minimality of $R_i$ and $L_j$.
In other words, $S$ acts on the right on each $R_i$ and on the left on each $L_j$.
In fact, $H_{ij}s=H_{il}$ for some $l\in J$; i.e.\ $S$ acts on the set of $\eh$-classes inside an $\ar$-class. 
When $S=K,$ the above facts may be easily seen from the classic structure theorem outlined below.

Let $T$ be a semigroup, let $I$ and $J$ be two index sets, and let $P=(p_{j,i})$ be a $J\times I$ matrix with entries from $T.$
The \emph{Rees matrix semigroup} $S=\mathcal{M}[T;I,J;P]$ is the set $I\times T\times J$ with multiplication
\[(i,u,j)(k,v,m)=(i,up_{j,k}v,m).\]
The Rees-Suschkewitsch Theorem \cite[Theorem 3.3.1]{Howie:1995} states that a semigroup $S$ is completely simple if and only if it is isomorphic to some $\mathcal{M}[G;I,J;P]$ where $G$ is a group.
Furthermore, in this situation $P$ can be chosen to be in \emph{normal form}, i.e.\ to satisfy
$p_{1,i}=p_{j,1}=1_G$ for all $i\in I$ and $j\in J$; here $1\in I\cap J$ should be viewed as an arbitrary fixed element of both $I$ and $J.$

{\em Semigroup actions} are representations of semigroups by transformations of sets.  More precisely, a \emph{right action} of a semigroup $S$ on a set $A$
is a map $A\times S\rightarrow A$, $(a, s)\mapsto as,$
such that $(as)t=a(st)$ for all $a\in A$ and $s, t\in S$.
If $S$ is a monoid and $a1=a$ for all $a\in A,$ then we have a \emph{monoid action}.  In either case we say that $A$ is a {\em right $S$-act}.  If $A$ is a right $S$-act, there is a natural associated monoid action of $S^1$ on $A$; we will make use of this association without further comment throughout.
  
Let $A$ be a right $S$-act.  A subset $B$ of $A$ is called a {\em subact} of $A$ if $bs\in B$ for all $b\in B$ and $s\in S$; that is, $B=BS^1.$
A subset $U$ of a $A$ is a {\em generating set} for $A$ if $A=US^1;$ $A$ is said to be {\em finitely generated} if it has a finite generating set and {\em monogenic} if it has a one-element generating set.
An equivalence relation $\rho$ on $A$ is a {\em congruence} if $(a, b)\in\rho$ implies $(as, bs)\in\rho$ for all $s\in S$. 
For $X\subseteq A\times A,$ the {\em congruence generated by} $X$ is the smallest congruence on $A$ containing $X$; we denote this congruence by $\langle X\rangle.$ 

\begin{Lem}\label{lem:sequence}\cite[Lemma I.\,4.\,37]{kkm:2000}
Let $S$ be a semigroup, let $A$ be a right $S$-act, and let $X$ be a subset of $A\times A.$  For any $a, b\in A,$ we have $(a, b)\in\l X\r$ if and only if either $a=b$ or there exists a sequence $$a=x_1s_1,\ y_1s_1=x_2s_2,\ \dots, \ y_ns_n=b$$ where $(x_i, y_i)\in X$ or $(y_i, x_i)\in X,$ and $s_i\in S^1,$ for all $i\in\{1, \dots, n\}.$
\end{Lem}

A sequence of the form given in Lemma~\ref{lem:sequence} is referred to as an $X$-\emph{sequence} of \emph{length} $n$ from $a$ to $b$;
if $a=b$ we say that there is an {\em $X$-sequence of length $0$
from $a$ to $b$}.

Every semigroup $S$ is a right $S$-act via right multiplication.
The subacts of this right $S$-act are precisely the right ideals of $S,$ and its congruences are the right congruences on $S.$  
Left/right dualising the preceding discussion, we arrive at the notion of \emph{left} semigroup acts and their basic properties.

\section{Pseudo-finiteness: Introduction}
\label{sec:pfintro}

Previously the notion of pseudo-finiteness has only been considered for semigroups.  However, we will see it is both natural and useful to define this notion for $S$-acts as well.
 
\subsection{Fundamental definitions}
\label{subsec:defn}

Let $S$ be a semigroup and let $A$ be a right $S$-act.  Consider a set $X\subseteq A\times A$ such that $\omega_A=\l X\r$, where $\omega_A$ denotes the universal relation on $A.$
For any $a, b\in A$ let $d_X(a, b)$ denote the smallest $n\in\mathbb{N}^0$ such that there is an $X$-sequence of length $n$ from $a$ to $b.$ 
It is easy to see that $d_X : A\times A\rightarrow \N^0$ is a metric.

\begin{Def}
Let $S$ be a semigroup and let $A$ be a right $S$-act.
\begin{itemize}[leftmargin=*]
\item If $\omega_A=\l X\r,$ we call the diameter of the metric space $(A, d_X)$ the {\em $X$-diameter} of $A$ and denote it by $D(X, A)$; in other words, \[D(X, A)=\sup\{d_X(a, b)\::\: a, b\in A\}.\]
\item If $\omega_A$ is finitely generated, we define the {\em diameter} of $A$ to be
$$D(A)=\min\{D(X, A) : \omega_A=\langle X\rangle, \ |X|<\infty\}.$$
\item We say that $A$ is {\em pseudo-finite} if it has finite diameter.
\end{itemize}
\end{Def}

Note that if $X$ and $Y$ are two finite generating sets for $\omega_A,$ then $D(X, A)$ is finite if and only if $D(Y, A)$ is finite; the proof of this fact is essentially the same as that of \cite[Lemma 2.5]{Dandan:2019}. 

\begin{Def} 
Let $S$ be a semigroup.
\begin{itemize}[leftmargin=*]
\item If $\omega_S=\l X\r,$ the {\em right $X$-diameter} of $S,$ denoted by $D_r(X, S),$ is the $X$-diameter of $S$ considered as a right $S$-act.
\item If $\omega_S$ is finitely generated, the {\em right diameter} of $S,$ denoted by $D_r(S),$ is the diameter of $S$ considered as a right $S$-act.
\item We say that $S$ is {\em right pseudo-finite} if it has finite right diameter (or, equivalently, $S$ is pseudo-finite as a right $S$-act).
\end{itemize}
\end{Def}

We dually define the notions of {\em left $X$-diameter} of $S$ and {\em left diameter} of $S,$ denoted by $D_l(X, S)$ and $D_l(S),$ respectively, and the notion of $S$ being {\em left pseudo-finite}.  
In a subsequent paper we will explore the notion of diameter in more detail, but it is convenient to have this terminology to draw upon here. 

\subsection{Finiteness conditions}
\label{subsec:fc}

It is clear that being right pseudo-finite is a semigroup {\em finiteness condition}, in the sense that every finite semigroup is right pseudo-finite.
In fact, for some classes of semigroups right pseudo-finiteness is \emph{equivalent} to being finite.
Most notably, this is the case for groups, as noted in \cite{Dandan:2019},  referring back to \cite{White:2017}.
In fact, a more general result is stated at the end of Section 4 of \cite{White:2017}, which we prove here for completeness.

A semigroup $S$ is said to be {\em weakly left cancellative} if for any $a, b\in S$ the set $\{s\in S^1 : a=bs\}$ is finite.  {\em Weakly right cancellative semigroups} are defined dually.  The class of weakly left cancellative semigroups includes all left cancellative semigroups (a semigroup $S$ is {\em left cancellative} if $ab=ac$ implies $b=c$ for all $a, b, c\in S$) and hence all groups.  

\begin{Prop}\label{prop:wlc}\cite[Section 4]{White:2017}
A weakly left cancellative semigroup $S$ is right pseudo-finite if and only if it is finite.
\end{Prop}

\begin{proof}
We have already remarked that being right pseudo-finite is a finiteness condition, so we only need to prove the direct implication.  Let $\emptyset\neq X\subseteq S$ be a finite generating set for $\omega_S$, and let $D_r(X, S)=n.$  For $c, d\in S,$ denote the set $\{s\in S^1 : c=ds\}$ by $[c : d].$  Each $[c : d]$ is finite by assumption.  Fix $b\in S.$  We define sets $U_i\subseteq S^1$ recursively as follows: 
$$U_1=\bigcup_{x\in X}[b : x],\; U_i=\bigcup_{x, y\in X}\bigcup_{u\in U_{i-1}}[yu : x]\;\, (i\geq 2).$$
Since $X$ is finite, by an easy induction argument we have that each $U_i$ is finite.  Let $U=\bigcup_{i=1}^nU_i,$ and let $V$ denote the finite set $XU.$  We claim that $S=V.$  Indeed, let $a\in S.$  Then there exists an $X$-sequence
$$b=x_1s_1, \ y_1s_1=x_2s_2,\ \dots,\ y_ks_k=a$$
where $k\leq n.$
We have that $s_1\in[b : x_1]\subseteq U_1,$ and hence $s_2\in[y_1s_1 : x_2]\subseteq U_2.$  Continuing in this way, we deduce that $s_k\in U_k\subseteq U,$ and hence $a=y_ks_k\in V,$ as required.
\end{proof}

The next result, following similar lines, is framed in terms of the so-called Green's \linebreak
$\ast$-equivalences.
The equivalence relation $\el^{\ast}$ on a semigroup $S,$ introduced in \cite{McAlister:1976}, is defined by the rule that $(a, b)\in\el^{\ast}$ if and only $a, b$ are $\el$-related in some oversemigroup $T.$  We say that $S$ is {\em $\el^{\ast}$-simple} if it has a single $\el^{\ast}$-class; {\em $\ar^{\ast}$-simple} semigroups are defined analogously.

\begin{Prop}\label{prop:L*-simple}
An $\el^{\ast}$-simple semigroup $S$ is right pseudo-finite if and only if it is finite.
\end{Prop}

\begin{proof}
Again, we just prove the direct implication.  Since $S$ is $\el^{\ast}$-simple, by \cite[Theorem 1]{Pastijn:1975} 
there exists an oversemigroup $T$ such that $S$ is contained in a single $\el$-class of $T.$
(One can take $T$ to be the dual of the full transformation monoid on $S^1,$ in which maps are composed from right to left.)

Now let $n=D_r(S),$ and let $X\subseteq S$ be a finite generating set for $\omega_S$ such that $D_r(X, S)=n.$  For each pair $x, y\in X,$ since $x$ and $y$ are $\el$-related in $T$ we can choose $\alpha(x,y)\in T$ such that $x=\alpha(x,y)y.$  Fix $b\in S.$  The set
$$U=\{\alpha(x_1,y_1)\dots\alpha(x_k,y_k)b : x_i, y_i\in X, k\leq n\}\subseteq T$$
 is finite since it consists of products of a finite number of elements of length at most $n+1.$  We claim that $S\subseteq U.$  Indeed, for any $a\in S$ there exists an $X$-sequence
$$a=x_1s_1,\ y_1s_1=x_2s_2,\ \dots,\ y_ks_k=b$$
where $k\leq n.$
Letting $\alpha_i=\alpha(x_i, y_i),$ we have that
$$a=x_1s_1=\alpha_1y_1s_1=\alpha_1x_2s_2=\alpha_1\alpha_2y_2s_2=\cdots
=\alpha_1\dots\alpha_ky_ks_k=\alpha_1\dots\alpha_kb\in U,$$
as required.
\end{proof}

\begin{Rem}
There is an intriguing connection between Propositions \ref{prop:wlc} and \ref{prop:L*-simple}.
On the one hand, there are considerable similarities in the structure of the proofs, even though they deal with fairly different sets of assumptions. On the other hand, if the single $\el$-class of $T$ in the proof of Proposition \ref{prop:L*-simple} happens to be the $\el$-class of the identity, then this implies that $S$ is left cancellative, thus recovering a special case of 
Proposition \ref{prop:wlc}.
\end{Rem}

\subsection{Diagonal acts}
\label{subsec:diagonal}

Given a semigroup $S,$ one can define a right action of $S$ on the set $S\times S$ by $(a, b)c=(ac, bc)$ for all $a, b, c\in S.$  With this action, $S\times S$ is called the {\em diagonal right $S$-act}.  The {\em diagonal left $S$-act} is defined dually.  Diagonal acts first appear, implicitly, in the work of Bulman-Fleming and McDowell \cite{Bulman:1989}.  They were formally defined and studied by Robertson et al.\ in \cite{Robertson:2001}, and the same authors then made use of this notion in relation to wreath products \cite{Robertson:2002}.  
The importance of diagonal acts for the theory of right pseudo-finite semigroups is encapsulated in the following result.

\begin{Prop}\label{prop:diagonal,diameter}
For a non-trivial semigroup $S,$ the diagonal right $S$-act is finitely generated if and only $S$ has right diameter 1.  In particular, if the diagonal right $S$-act is finitely generated, then $S$ is right pseudo-finite.
\end{Prop}

\begin{proof}
Suppose first that the diagonal right $S$-act is generated by a finite set $U.$ This means that for any $a, b\in S$ there exist $(u, v)\in U$ and $s\in S^1$ such that $(a, b)=(u, v)s$.  But then $a=us, vs=b$ is a $U$-sequence of length $1$, and hence $D_r(S)=1.$
Conversely, if $D_r(S)=1,$ let $X$ be a generating set for $\omega_S$ such that $D_r(X,S)=1.$  Setting
$$U=\{(x, y)\::\: (x, y)\in X \text{ or } (y, x)\in X\},$$ it follows readily that $S\times S=US^1$.
\end{proof}

Gallagher \cite{Gallagher:2006, Gallagher:2005} systematically studied finitely generated diagonal acts, a class that turns out to be quite rich and rather curious.  
As a source of examples, we summarise his findings regarding certain monoids of transformations and binary relations:

\begin{Thm}\label{thm:diag}\cite[Table 1]{Gallagher:2005}
Let $X$ be an infinite set.
\begin{thmenumerate}
\item The monoid $\mathcal{B}_X,$ consisting of all binary relations on $X,$ has cyclic diagonal right act and cyclic diagonal left act.
\item The full transformation monoid $\mathcal{T}_X$ has cyclic diagonal right act and cyclic diagonal left act.
\item The partial transformation monoid $\mathcal{P}_X$ has cyclic diagonal right act and cyclic diagonal left act.
\item The monoid $\mathcal{F}_X,$ consisting of all transformations on $X$ whose kernel classes are finite, has cyclic diagonal right act, but its diagonal left act is not finitely generated.
\end{thmenumerate}
\end{Thm}

\subsection{Basic properties}
We begin this subsection by remarking that, given an $S$-act $A,$ any finite generating set for $\omega_A$ is contained in one of the form $X\times X$ for some finite set $X\subseteq A.$  We shall often abuse terminology by saying that {\em $\omega_A$ is generated by $X$}, by which we mean that $\omega_A$ is generated by $X\times X$.  Similarly, we shall speak of the $X$-diameter of $A,$ meaning the ($X\times X$)-diameter.

We now establish some basic results concerning pseudo-finiteness of acts and semgroups.  

\begin{Lem}\label{lem:pffg} 
Let $S$ be a semigroup.  Every pseudo-finite right $S$-act is finitely generated.  In particular, if $S$ is right pseudo-finite then it is finitely generated as a right ideal.
\end{Lem}

\begin{proof} 
Let $A$ be a pseudo-finite right $S$-act.  If $A$ is trivial then it is certainly finitely generated, so suppose that $A$ has at least two elements.  There exists $X\subseteq A$ such that $A$ has finite $X$-diameter.  Let $a\in A$, and pick any $b\in A$, $b\neq a$.  Then there exists an $X$-sequence of positive length connecting $a$ to $b,$ so that $a=xs$ for some $x\in X$ and $s\in S^1.$  Thus $A$ is generated by $X.$ 
\end{proof}

\begin{Lem}\label{lem:subact} 
Let $S$ be a semigroup.  Let $A$ be a finitely generated right $S$-act and let $B$ be a subact of $A.$  If $B$ is pseudo-finite then so is $A.$
\end{Lem}

\begin{proof} 
We have that $A=US^1$ for some finite set $U\subseteq A.$  Suppose that $\omega_B=\langle X\rangle$ for some finite set $X\subseteq B.$  Since $B$ is pseudo-finite, the $X$-diameter $D(X, B)$ is finite.  For any $a\in A,$ there exist some $u\in U$ and $s\in S^1$ such that $a=us.$  Then $xs\in B$ for any $x\in X.$  It follows that $\omega_A=\l Y\r,$ where $Y=X\cup U,$ and that $D(Y, A)\leq D(X, B)+2.$  Thus $D(A)\leq D(Y, A)$ is finite, and hence $A$ is pseudo-finite.
\end{proof}

\begin{Lem}\label{lem:actquotient} 
Let $S$ be a semigroup.  Let $A$ be an $S$-act and let $B$ be a homomorphic image of $A.$  If $A$ is pseudo-finite then so is $B.$
\end{Lem}

\begin{proof}
Let $\omega_A=\langle X\rangle$ for some finite set $X\subseteq A.$  Since $A$ is pseudo-finite, the $X$-diameter $D(X, A)$ is finite.  Let $\theta : A\to B$ be a surjective homomorphism, and let $Y=X\theta.$  Applying $\theta$ to any $X$-sequence yields a $Y$-sequence of the same length.  It follows that $\omega_B=\langle Y\rangle$ and that $D(Y, B)\leq D(X, A).$  Thus $D(B)\leq D(Y, B)$ is finite, and hence $B$ is pseudo-finite.
\end{proof}

\begin{Lem}
\label{Lem:STacts}
Suppose that $S$ is a subsemigroup of $T$, and let $A$ be a $T$-act.
If $A$ is pseudo-finite as an $S$-act, then it is also pseudo-finite as a $T$-act.
\end{Lem}

\begin{proof}
For any $X\subseteq A\times A$, every $X$-sequence where $A$ is regarded as an $S$-act is also an $X$-sequence with $A$ regarded as a $T$-act. 
\end{proof}

Turning to right pseudo-finiteness of semigroups, a similar argument to that of Lemma \ref{lem:actquotient} proves:

\begin{Lem}\label{lem:quotient}\cite[Proposition 4.1]{Dandan:2019}
Let $S$ be a semigroup and let $T$ be a homomorphic image of $S.$  If $S$ is right pseudo-finite then so is $T.$
\end{Lem}

\begin{Lem}\label{lem:sgpvsmonoid} 
If $S$ is right pseudo-finite semigroup then so is $S^1.$ 
\end{Lem}

\begin{proof} 
The $S^1$-act $S^1$ contains $S$ as a subact. This subact is pseudo-finite by Lemma \ref{Lem:STacts}, and hence 
$S^1$ is right pseudo-finite by Lemma \ref{lem:subact}.
\end{proof} 

The converse of the previous lemma is not true.  
For instance, let $S$ be any semigroup with zero that is not finitely generated as a right ideal (such as the infinite semilattice $S$ with zero in which $st=0$ for any $s\neq t$).  Then $S$ is not right pseudo-finite by Lemma \ref{lem:pffg}.  However, $S^1$ is right pseudo-finite by \cite[Corollary 2.15]{Dandan:2019}, since it is a monoid with zero.

\section{Pseudo-finiteness and Ideals}
\label{sec:pf,ideals}

We saw in Subsection \ref{subsec:fc} that for certain classes of semigroups, notably groups, right pseudo-finiteness is equivalent to finiteness.  On the other hand, as noted before, any monoid $S$ with a zero is right pseudo-finite.  We have already remarked that having a zero is the same as having a trivial minimal ideal. 
It is relatively easy to see that the assumption that $S$ be a monoid can be weakened to $S$ being finitely generated as a right ideal, and the assumption of the existence of a zero can be replaced with a finite minimal ideal; see also \cite[Corollary 8.2, Remark 8.9]{Dandan:2019}.

The foregoing discussions point to the following natural question: under what conditions, and in what ways, the presence of a minimal ideal implies right pseudo-finiteness of the semigroup.  This will be one of the guiding questions throughout this paper.  The following easy general result, which relates right pseudo-finiteness of a monoid with pseudo-finiteness of its right ideals and acts, will prove invaluable in these considerations.

\begin{Prop}\label{prop:rightideals}
The following are equivalent for a monoid $S$:
\begin{thmenumerate}
\item $S$ is right pseudo-finite;
\item $S$ has a right ideal that is pseudo-finite as a right $S$-act;
\item every principal right ideal of $S$ is pseudo-finite as a right $S$-act;
\item every finitely generated right ideal of $S$ is pseudo-finite as a right  $S$-act;
\item every monogenic right $S$-act is pseudo-finite;
\item every finitely generated right $S$-act is pseudo-finite.
\end{thmenumerate} 
\end{Prop}

\begin{proof} 
The implications (6)$\Rightarrow$(4)$\Rightarrow$(2)
and (6)$\Rightarrow$(5)$\Rightarrow$(3)$\Rightarrow$(2) are straightforward, and an application of Lemma \ref{lem:subact} yields (2)$\Rightarrow$(1).

(1)$\Rightarrow$(6).  Let $A$ be a finitely generated $S$-act.  We claim that the diameter $D(A)$ of $A$ is at most $2D_r(S)+1,$ which is finite since $S$ is right pseudo-finite.  Let $X\subseteq S$ be a finite generating set for $\omega_S$ such that $D_r(X, S)=D_r(S).$  Now let $U$ be a finite generating set for $A$ and put $V=UX^1.$  Let $a, b\in A.$  Then $a=us$ and $b=vt$ for some $u, v\in U$ and $s, t\in S.$  By assumption we have an $X$-sequence
$$s=x_1s_1,\ y_1s_1=x_2s_2,\ \dots,\ y_ks_k=1$$ 
where $k\leq D_r(S).$  Hence, we have a $V$-sequence
$$a=(ux_1)s_1,\ (uy_1)s_1=(ux_2)s_2,\ \dots, \ (uy_k)s_k=u$$ 
from $a$ to $u.$  Similarly, there exists a $V$-sequence from $b$ to $v$ of length at most $D_r(S).$  Since $u, v\in V,$ we conclude that $a$ and $b$ can be connected by a $V$-sequence of length at most $2D_r(S)+1,$ as required.
\end{proof}

Combining Lemma \ref{Lem:STacts} and Proposition \ref{prop:rightideals}, we have:

\begin{Cor}
\label{cor:rightideal}
Let $S$ be a monoid and let $I$ be a right ideal of $S.$  If $I$ is right pseudo-finite (as a semigroup), then $S$ is right pseudo-finite.
\end{Cor}

Consider a right ideal $I$ of a monoid $S.$  If $I$ has an identity, then it is a retract of $S.$  Indeed, letting $1_I$ denote the identity of $I,$ define a map $\theta : S\to I$ by $s\theta=1_Is.$  For any $s, t\in S,$ we have
$$(st)\theta=1_I(st)=(1_Is)t=\bigl((1_Is)1_I\bigr)t=(1_Is)(1_It)=(s\theta)(t\theta),$$
so $\theta$ is a homomorphism.  Clearly $\theta|_I$ is the identity map on $I,$ so $\theta$ is a retraction, as required.  (In fact, the converse also holds: if $I$ is a retract of $S$ via a retraction $\theta : S\to I,$ then $I$ has identity $1_S\theta.$)  From this observation and Lemma \ref{lem:quotient}, along with Corollary \ref{cor:rightideal}, we deduce:

\begin{Cor}
\label{cor:monoidideal}
Let $S$ be a monoid, and let $I$ be a right ideal of $S$ that has an identity.  Then $S$ is right pseudo-finite if and only if $I$ is right pseudo-finite.
\end{Cor}

Going in the converse direction, we may wonder in what situations right pseudo-finiteness of a semigroup implies the existence of minimal ideals, or even minimal ideals of a certain kind.  This is certainly the case in all instances where right pseudo-finiteness implies finiteness, as discussed in Subsection \ref{subsec:fc}, since we noted earlier that a finite semigroup must possess a completely simple minimal ideal.
Also, if $S$ is right pseudo-finite with exactly one minimal left ideal $L$ and exactly one minimal right ideal $R$, then by \cite[Theorem 4.2]{Clifford:1948} we have that $L=R$ is the minimal ideal of $S$ and is also a group.
It follows from Corollary \ref{cor:monoidideal} and Proposition \ref{prop:wlc} that this group must be finite. 
 
Returning to various natural semigroups with cyclic diagonal acts encountered in Theorem \ref{thm:diag}, we remark that the monoids $\mathcal{B}_X$, $\mathcal{P}_X$ and $\mathcal{T}_X$ each have a completely simple minimal ideal.  Indeed, the former two both contain a zero element, and the minimal ideal of $\mathcal{T}_X$ is a right zero semigroup, consisting of all the constant maps on $X$ (this minimal ideal is infinite since $X$ is infinite).
The monoid $\mathcal{F}_X$ turns out to be {\em bisimple}, meaning that it has a single $\mathcal{D}$-class, and hence regular (since any bisimple monoid is regular).  (The proof that $\mathcal{F}_X$ is bisimple is essentially the same as the proof that the similarly-defined monoid $\mathcal{M}(X)$ is bisimple; see \cite[Section 8.6]{Clifford:1967}.  We note that $\mathcal{M}(X)=\mathcal{F}_X$ when $X$ is countable.)  The monoid $\mathcal{F}_X$ is not completely simple; indeed, it can be easily deduced from the Rees-Suschkewitsch representation, given in Section \ref{sec:prelim}, that a monoid is completely simple if and only if it is a group, and $\mathcal{F}_X$ is certainly not a group.  Thus there exist right pseudo-finite (regular) monoids with minimal ideals that are not completely simple.

Given any infinite set $X,$ the Baer--Levi semigroup 
$$\mathcal{BL}_X=\{\alpha\in\mathcal{T}_X : \alpha\text{ is injective}, |X\!\setminus\!X\alpha|=|X|\}$$
is a right simple, right cancellative semigroup without idempotents (so certainly not competely simple) \cite[Theorem 8.2]{Clifford:1967}, and is right pseudo-finite \cite[Remark 7.3]{Miller:2020}.  It can be easily shown that $\mathcal{BL}_X$ is the minimal ideal of the monoid $\Inj_X$ of all injective mappings on $X.$  Thus, by Corollary \ref{cor:rightideal}, we have:

\begin{Prop}
For any infinite set $X,$ the monoid $\Inj_X$ is right pseudo-finite.
\end{Prop}

\begin{Rem} Let $X$ be an infinite set.  The monoid $\Inj_X$ is $\ar^{\ast}$-simple since it coincides with the $\ar$-class of the identity of $\mathcal{T}_X.$  It follows from the dual of Proposition \ref{prop:L*-simple} that no infinite subsemigroup of $\Inj_X$ is left pseudo-finite.
\end{Rem}
 
From the preceding discussion a potentially intricate landscape begins to emerge, relating the property of pseudo-finiteness with the existence and/or nature of minimal ideals. The aim of this paper is to provide an in-depth exploration of this landscape.

\section{Completely Simple Minimal Ideals}\label{sec:csmi}

For the remainder of the paper we focus on monoids, since if $S$ is right pseudo-finite/has a minimal ideal, then the same properties are true of $S^1$.  In this section we discuss the relationship between the property of being right pseudo-finite and the existence of a completely simple minimal ideal.  We first establish a result that characterises right pseudo-finiteness in the presence of a completely simple minimal ideal.  We then discover various classes of semigroup for which being right pseudo-finite implies the existence of such an ideal.

The following result provides two necessary and sufficient conditions for a monoid with a completely simple minimal ideal to be right pseudo-finite.  The first is new, whereas the second was indicated in \cite[Remark 8.9]{Dandan:2019} where it was noted that the results of that section, which concerned the universal left congruence being finitely generated, could be modified to the (left) pseudo-finite case.  In fact, the modifications in this instance are significant, and we give a direct argument below. 

The statement features the action of a semigroup on the $\eh$-classes in a minimal right ideal; this was introduced in Section \ref{sec:prelim}.

\begin{Thm}\label{thm:csmi} 
Let $S$ be a monoid with a completely simple minimal ideal $K.$  Then the following three statements are equivalent.
\begin{thmenumerate}
\item $S$ is right pseudo-finite.
\item $S$ satisfies the following two conditions:
\begin{enumerate}
\item there exists a (completely simple) left ideal $K_0$ of $K$ such that $K_0$ is the union of finitely many $\el$-classes and $K_0^1$ is right pseudo-finite;
\item for any $\ar$-class $R$ of $K,$ the right $S$-act $R/\eh$ is pseudo-finite.
\end{enumerate}
\item $S$ satisfies the following two conditions:
\begin{enumerate}
\item there exists a left ideal $K_0$ of $K$ such that $K_0$ is the union of finitely many $\el$-classes and any maximal subgroup $G=H_e$ of $K_0$ has finite ($F\cup V$\!)-diameter, where $F\subseteq G$ is finite and 
$$V=\{fg : f, g \in E(K_0), f\,\ar\,e\,\el\,g\}\subseteq G;$$
\item for any $\ar$-class $R$ of $K,$ the right $S$-act $R/\eh$ is pseudo-finite.
\end{enumerate}
\end{thmenumerate}
\end{Thm}

\begin{proof}
(1)$\Rightarrow$(2).  We first prove that (2a) holds.  Let $X\subseteq S$ be a finite generating set for $\omega_S,$ and let $n=D(X, S).$  Fix an idempotent $e\in K.$  We may assume that $e\in X.$  Let $V=\{ex : x\in X\}\subseteq R_e,$ and let 
$$K_0=\bigcup_{v\in V}L_v.$$  
Let $$Y=\{1, ex : x\in X\}\cup\bigl(E(K_0)\cap R_e\bigr)\subseteq(K_0\cap R_e)^1.$$
Clearly $Y$ is finite.  We claim that the $Y$-diameter of $K_0^1$ is no more than $2n+3.$  Indeed, let $a, b\in K_0.$  Let $f$ be the idempotent in the $\eh$-class of $ea,$ and let $g$ be the idempotent in the $\eh$-class of $eb.$  Then $f, g\in E(K_0)\cap R_e\subseteq Y.$  Now, there exists an $X$-sequence
$$a=x_1s_1,\ y_1s_1=x_2s_2,\ \dots,\ y_ks_k=f$$
in $S,$ where $k\leq n.$
Therefore, we have a $Y$-sequence
$$a=1a,\ ea=eaf=(ex_1)(s_1f),\ (ey_1)(s_1f)=(ex_2)(s_2f),\ \dots,\ (ey_k)(s_k)f=ef^2=f$$
in $K_0$ that has length $k+1.$  Similarly, there exists a $Y$-sequence of length at most $n+1$ from $b$ to $g.$  Since $f, g\in Y,$ we conclude that there exists a $Y$-sequence of length at most $2n+3$ from $a$ to $b,$ as required.\par
For (2b), let $R$ be any $\ar$-class of $K.$  Then $R$ is a pseudo-finite as a right $S$-act by Proposition \ref{prop:rightideals}, and hence the quotient $R/\mathcal{H}$ is pseudo-finite by Lemma \ref{lem:actquotient}.\par
(2)$\Rightarrow$(3).  Condition (3b) is identical to (2b), so we just need to prove that (3a) holds.  Let $K_0$ be as given in (2a).  In particular, $K_0$ is the union of finitely many $\el$-classes.  Consider a maximal subgroup $G=H_e$ of $K_0.$  Let $T=K_0^1.$  Since $T$ is right pseudo-finite, there exists a finite set $Y\subseteq T$ such that $\omega_T=\langle Y\rangle$ and the $Y$-diameter of $T$ is finite, say $n.$
Let $F=\{e,\, eye: y\in Y\},$ and let $X=V\cup F$ where $V$ is as given in the statement.  Clearly $F$ is finite, but $V$ may be infinite.  We claim that $\omega_G=\langle X\rangle$ and that $X$-diameter of $G$ is no greater than $3n.$  Indeed, let $u, v\in G.$  Then there exists a $Y$-sequence
\begin{equation}\label{eq:1} 
u=x_1t_1,\, y_1t_1=x_2t_2,\hdots, y_kt_k=v
\end{equation}
in $T,$ where $k\leq n$.  Let $x_i'=ex_ie, y_i'=ey_ie, t_i'=et_ie,$ let $e_i, f_i, g_i$ be the idempotents in the $\eh$-classes of
$ex_i, ey_i, t_ie,$ respectively, and let $a_i=e_ig_i$ and $b_i=f_ig_i.$  The elements are arranged in the following egg-box pattern.

\[\begin{array}{ |c|c|c|c|c|c|c|}
\hline
\begin{array}{c}
e, x_i', y_i'\\
t_i^{\prime}, a_i, b_i\end{array}&\cdots&ex_i, e_i&\phantom{xx}&ey_i, f_i&\cdots&\phantom{xx}\\ \hline
\vdots &&&&&&
\\ \hline
t_ie, g_i&&&&&&\\ \hline
\vdots&&&&&&\\ \hline
&&&&&&\\ \hline
\end{array}
\vspace{0.7em}\] 
Note that $x_i^{\prime}, y_i^{\prime}\in F$ and $a_i, b_i\in V.$  We claim that we have a sequence
\begin{equation}\label{eq:2} 
u=x_1'a_1t_1',\,  y_1'b_1t_1'=x_2'a_2t_2',\hdots, y_k'b_kt_k'=v.
\end{equation}
To see this, observe that
\[x_i'a_it_i'=(ex_ie)(e_ig_i)(et_ie)=(ex_i)(ee_i)(g_ie)(t_ie)=(ex_i)e_ig_i(t_ie)=(ex_i)(t_ie).\]
Similarly, we have $y_i'a_it_i'=(ey_i)(t_ie).$  Thus, multiplying the sequence \eqref{eq:1} both on the left and right by $e$ yields the sequence \eqref{eq:2}.  Now, for each $i\in\{1, \dots, k\},$ there exists an $X$-sequence
$$x_i'a_it_i'=x_i'(a_it_i'),\, e(a_it_i')=a_it_i',\, b_it_i'=e(b_it_i'),\, y_i'(b_it_i')=y_i'b_it_i',$$
which has length 3.  We conclude that there exists an $X$-sequence of length no greater than $3n$ from $u$ to $v,$ as required.

(3)$\Rightarrow$(1).  Fix $e\in K_0$ and let $R=R_e.$  By Proposition \ref{prop:rightideals}, it suffices to prove that $R$ is pseudo-finite as a right $S$-act.  Let $G=H_e.$  By (3a), $G$ has finite $(F\cup V)$-diameter, say $n,$ where $F$ and $V$ are as given in the statement.  By (3b), the quotient $A=R/\eh=\{ [a]_{\mathcal{H}}\, :\, a\in R\}$ is pseudo-finite.  Let $\omega_A=\langle Y\rangle$ for some finite set $Y\subseteq A,$ and let $m$ be the $Y$-diameter of $A.$  For each $y\in Y$ choose $x_y\in R$ such that $y=[x_y]_{\mathcal{H}},$ and let $X=\{x_y : y\in Y\}.$  We claim that $\omega_R$ is generated by the finite set
$$Z=F\cup\bigl(E(K_0)\cap R\bigr)\cup X,$$
and that the $Z$-diameter of $R$ is no greater than $2n(m+1)+m.$\par
We first claim that for any $u, v\in R$ such that $u\,\eh\,v,$ there exists a $Z$-sequence of length no greater than $2n$ from $u$ to $v.$
Indeed, let $u$ and $v$ be as given above.  If $u=v$ then we are done, so assume that $u\neq v.$  Let $h$ be the idempotent in $H_u=H_v.$  We have that $ue, ve\in G,$ so there exists an $(F\cup V)$-sequence
$$ue=u_1s_1, v_1s_1=u_2s_2, \dots, v_ks_k=ve$$
where $k\leq n.$  Since $eh=h$ and $uh=u, vh=v,$ multiplying the above sequence on the right by $h,$ we obtain an $(F\cup V)$-sequence
$$u=u_1s_1h, v_1s_1h=u_2s_2h, \dots, v_ks_kh=v.$$
If $u_i, v_i\in F$ for all $i\in\{1, \dots, k\},$ then we have an $F$-sequence from $u$ to $v,$ and we are done.  So suppose otherwise, and consider $(w, z)\in\{(u_i, v_i), (v_i, u_i)\}$ such that $w\in V.$  Then $w=fg$ where $f, g \in E(K_0)$ and $f\,\ar\,e\,\el\,g.$  Since $e, f\in E(K_0)\cap R,$ we have a $Z$-sequence
$$ws_ih=f(gs_ih),\, e(gs_ih)=es_ih.$$
If $z\in V$ then, by the same argument, there exists a $Z$-sequence of length 1 from $zs_ih$ to $es_ih.$  Otherwise, if $z\in F,$ then clearly we have a $Z$-sequence of length 1 from $zs_ih$ to $es_ih.$  It follows that there is a $Z$-sequence of length 2 from $ws_ih$ to $zs_ih.$  We conclude that there is a $Z$-sequence of length no greater than $2n$ from $u$ to $v,$ establishing the claim.\par
Now let $a, b\in R.$  Then $[a]_{\mathcal{H}}, [b]_{\mathcal{H}}\in A,$ so there exist a $Y$-sequence
$$[a]_{\mathcal{H}}=y_1t_1, z_1t_1=y_2t_2, \dots, z_lt_l=[b]_{\mathcal{H}},$$
where $y_i, z_i\in Y,$ $t_i\in S^1$ and $l\leq m.$  Letting $x_i=x_{y_i}$ and $x_i^{\prime}=x_{z_i},$ we deduce that
$$a\,\eh\,x_1t_1,\, x_1^{\prime}t_1\,\eh\,x_2t_2,\, \dots, x_l^{\prime}t_l\,\eh\,b.$$
Note that $x_i, x_i^{\prime}\in X.$  By the above claim, for each pair 
$(u, v)$ in
$$\{(a, x_1t_1), (x_i^{\prime}t_i, x_{i+1}t_{i+1}), (x_l^{\prime}t_l, b) : 1\leq i\leq l-1\},$$ there exists a $Z$-sequence of length no greater than $2n$ from $u$ to $v.$  By interleaving these sequences with single steps from $x_it_i$ to $x_it_i^{\prime},$ we obtain a $Z$-sequence of length no greater than $2n(m+1)+m$ from $a$ to $b.$  This completes the proof.
\end{proof}

\begin{Cor}
\label{cor:csmi}
Let $S$ be a right pseudo-finite monoid with a completely simple minimal ideal $K.$  If $K$ has finitely many $\ar$-classes, then its maximal subgroups are finite.
\end{Cor}

\begin{proof}
By Theorem \ref{thm:csmi}, there exists a left ideal $K_0$ of $K$ such that $K_0$ is the union of finitely many $\el$-classes and any maximal subgroup $G=H_e$ of $K_0$ has finite ($F\cup V$\!)-diameter, where $F\subseteq G$ is finite and 
$$V=\{fg : f, g \in E(K_0), f\,\ar\,e\,\el\,g\}.$$
Since every maximal subgroup of $K$ is isomorphic to $G,$ it suffices to prove that $G$ is finite.
Since $K_0$ is completely simple (and hence regular), Green's relation $\ar$ on $K_0$ is the restriction of Green's relation $\ar$ on $K$ \cite[Proposition 2.4.2]{Howie:1995}.  Therefore, since $K$ has finitely many $\ar$-classes, so does $K_0.$  Since $K_0$ has finitely many $\el$-classes, we conclude that $K_0$ is the union of finitely many maximal subgroups.  Thus $E(K_0)$ is finite.  It follows that $V$ is finite.  Since $F$ is finite, we have that $F\cup V$ is finite, and hence $G$ is right pseudo-finite.  Then $G$ is finite by Proposition \ref{prop:wlc}.
\end{proof}

\begin{Rem}
\label{rem:csmi}
If a monoid $S$ has a completely simple minimal ideal $K$ whose maximal subgroups are finite, then $S$ clearly satisfies condition (3a) of Theorem \ref{thm:csmi} (where $K_0$ can be taken to be any $\el$-class of $K$), so $S$ is right pseudo-finite if and only if for any $\ar$-class $R$ of $K$ the right $S$-act $R/\eh$ is pseudo-finite.
\end{Rem}

\begin{Rem}
It is possible for a right pseudo-finite monoid to have a completely simple minimal ideal that has finitely many $\ar$-classes and infinitely many $\el$-classes.  Indeed, as discussed in Section \ref{sec:pfintro}, the full transformation monoid $\mathcal{T}_X$ on an infinite set $X$ is right pseudo-finite and has a minimal ideal that is an infinite right zero semigroup, which has a single $\ar$-class and infinitely many $\el$-classes.
\end{Rem}

Specialising Theorem \ref{thm:csmi} to completely simple semigroups with $1$ adjoined, we obtain:

\begin{Cor}
\label{Cor:cs1}
Let $K$ be a completely simple semigroup and let $S=K^1.$  Then $S$ is right pseudo-finite if and only if: 
\begin{thmenumerate}
\item $K$ has finitely many $\el$-classes; and 
\item any maximal subgroup $G=H_e$ of $K$ has finite ($F\cup V$\!)-diameter, where $F\subseteq G$ is finite and $V=\{fg : f, g \in E(K), f\,\ar\, e\,\el\, g\}.$
\end{thmenumerate}
\end{Cor}

\begin{proof}
Notice that the action of a completely simple semigroup $K$ on $R/\eh$, where $R$ is an $\ar$-class, satisfies the following property: if $[x]_{\mathcal{H}}, [y]_{\mathcal{H}}\in R/\eh$, then for any $s\in K$ we have
$[x]_{\mathcal{H}}s=[y]_{\mathcal{H}}s$. 
It follows that every equivalence relation on the $S$-act $R/\eh$ is a congruence, or, in other words, the congruence generated by a set is the smallest equivalence relation containing that set.
Hence the full congruence is finitely generated only if $R/\eh$ is finite.
The rest of the proof is a direct application of Theorem \ref{thm:csmi}.
\end{proof}


\begin{Rem}
As discussed in Section \ref{sec:prelim}, every completely simple semigroup can be represented as a Rees matrix semigroup $K=\mathcal{M}[G;I, J;P]$, where $G$ is a group and $P$ is normal, meaning that $p_{1,i}=p_{j,1}=1_G$ for all $i\in I$, $j\in J,$ where $1\in I\cap J.$
In this representation, asssuming without loss of generality that $e=(1, 1_G, 1),$ the set $\{fg : f, g \in E(K), f\,\ar\, e\,\el\, g\}$ corresponds to the set of entries $\{p_{j,i}\:: \: i\in I,\  j\in J\}$; see the proof of \cite[Theorem 3.2.3]{Howie:1995}.
If this set of entries comprises the whole of $G,$ and if $J$ is finite, Corollary \ref{Cor:cs1} implies that $S=K^1$ is right pseudo-finite.
Specialising further, if we take an infinite group $G$, take $I$ such that $|I|=|G|,$ set $J=\{1,2\}$, and populate the second row of $P$ with all the elements of $G,$ we obtain a right pseudo-finite semigroup with a completely simple ideal that has infinite maximal subgroups.
\end{Rem}

In what follows we consider some conditions on right pseudo-finite monoids that imply the existence of a completely simple minimal ideal.

\cite[Proposition 5.3]{Dandan:2019} provides necessary and sufficient conditions for an inverse monoid to be right pseudo-finite. An immediate consequence is:

\begin{Prop}
\label{prop:inv}
An inverse monoid is right pseudo-finite if and only if it has a minimal ideal that is a finite group.
\end{Prop}

A semigroup is said to be {\em completely regular} if it is a union of groups.  The class of completely regular semigroups includes completely simple semigroups, Clifford semigroups and bands.  By \cite[Corollary 8.3]{Dandan:2019} we have:

\begin{Prop}\label{prop:cr}
Every right pseudo-finite completely regular monoid has a completely simple minimal ideal.
\end{Prop}

A regular semigroup is said to be {\em orthodox} if its idempotents form a subsemigroup.  
We note that all inverse semigroups are orthodox, but the converse is not true; also, orthodox semigroups need not be completely regular, and vice versa.

The proof of Theorem~\ref{thm:orthodox} below makes use of some techniques from classical semigroup theory.  We explain them as we come across them, with the exception of the following construction of a semilattice of subsemigroups, which we will also need in Section~\ref{sec:mi}.

\begin{Def}\label{defn:slss} A semigroup $S$ is a {\em semilattice $Y$ of subsemigroups $S_{\alpha},\alpha\in Y$}, if (i) $S_\alpha\cap S_{\beta}=\emptyset$ for all $\alpha\neq \beta\in Y$; (ii) $S=\bigcup_{\alpha\in Y}S_\alpha$; and (iii) $S_{\alpha}S_{\beta}\subseteq S_{\alpha\beta}$ for all $\alpha,\beta\in Y$. 
\end{Def}

Notice that if $S$ is a semilattice $Y$ of subsemigroups, then $Y$ is a homomorphic image of $S.$  Thus, if $S$ is right pseudo-finite, then by Lemma~\ref{lem:quotient} so is $Y,$ and hence $Y$ is forced to have a zero by \cite[Proposition 5.3]{Dandan:2019}.  It is worth remarking that a completely regular semigroup is a semilattice of completely simple semigroups, which together with the foregoing remark yields Proposition~\ref{prop:cr}.  For orthodox semigroups we must work a little harder.

\begin{Thm}\label{thm:orthodox} 
Let $S$ be an orthodox monoid.  Then the following are equivalent:
\begin{thmenumerate}
\item $S$ is right pseudo-finite;
\item $S$ has a completely simple minimal ideal $K$ whose (maximal) subgroups are finite, and the right $S$-act $R/\eh$ is pseudo-finite for any $\mathcal{R}$-class $R$ of $K.$
\end{thmenumerate}
\end{Thm}

\begin{proof}
(1)$\Rightarrow$(2).  Denoting by $B$ the band of idempotents $E(S)$ of $S$, we have that $B$ is a semilattice $Y$ of rectangular bands $B_{\alpha}, \alpha\in Y$ \cite[Theorem 4.4.1]{Howie:1995}.  Moreover, $S/\gamma$ is an inverse monoid, where $\gamma$ is the least inverse congruence on $S.$  It follows from \cite[(6.2.5)]{Howie:1995} that $E(S/\gamma)\cong Y.$  Now, $S/\gamma$ is right pseudo-finite by Lemma \ref{lem:quotient}, and hence, by \cite[Proposition 5.3]{Dandan:2019}, the semilattice $Y$ has a least element $0.$  It follows that the rectangular band $B_0$ is the minimal ideal {\em of $B$}.  Fix an idempotent $e$ in $B_0.$  We claim that the $\el$-class $L_e$ of $S$ is a minimal left ideal.  Clearly it suffices to prove that $L_e$ is a left ideal.  So, let $a\in L_e$ and $s\in S.$  Then $a=ae$ and hence $sa=sae.$  Let $f$ be an idempotent such that $sa\,\el\,f.$  Then it follows that $f=fe.$  Consequently, by the minimality of $B_0,$ we have $f\in B_0.$  From $f=fe$ and the fact that $B_0$ is a rectangular band, we obtain $f\,\el\,e.$  It follows by transitivity that $sa\in L_e,$ as required.  A dual argument proves that the $\ar$-class $R_e$ is a minimal right ideal of $S.$  Since $S$ has both a minimal left ideal and a minimal right ideal, it has a completely simple minimal ideal, say $K.$


To prove that the maximal subgroups of $K$ are finite, we consider the map 
$$\phi_e : S\rightarrow S, a\mapsto eae.$$  
Since $K$ is completely simple and the minimal ideal of $S,$ it is clear that the image of $\phi_e$ is $H_e.$  Let $a, b\in S.$  Since $S$ is regular, there exist idempotents $g, h\in E(S)$ such that $ea\,\el\,g$ and $h\,\ar\,be.$  But then $g, h\in B_0.$  Since $B_0$ is a rectangular band, we have $geh=gh.$  Consequently, we have
\begin{align*}
(ab)\phi_e&=e(ab)e= (ea)(be)=(eag)(hbe)=(ea)(gh)(be)\\
&=(ea)(geh)(be)=(eag)e(hbe)=(ea)e(be)=(eae)(ebe)\\
&=(a\phi_e)(b\phi_e),
\end{align*}
so that $\phi_e$ is a homomorphism.  Thus $H_e$ is right pseudo-finite by Lemma \ref{lem:quotient}, and hence $H_e$ is finite by Proposition \ref{prop:wlc}.

By Theorem~\ref{thm:csmi}, the right $S$-act $R/\eh$ is pseudo-finite for any $\mathcal{R}$-class $R$ of $K.$  

(2)$\Rightarrow$(1).  This follows from Remark~\ref{rem:csmi}. 
\end{proof}

\begin{Ex}\label{ex:orthodox}
In the case of orthodox monoids we cannot make inferences about finiteness of the minimal ideal, or of its constituent $\ar$- or $\el$-classes. Indeed, let $S$ be any infinite right pseudo-finite orthodox monoid, e.g.\ an infinite group with a zero adjoined.
Let $T$ be the extension by constants of $S$; see \cite[p.\ 155]{Grillet:1995}. 
Then $T=S\cup I$, where $I=\{c_u\::\: u\in S\}$ is a right zero semigoup (i.e.\ $c_uc_v=c_v$), and $sc_u=c_u$, $c_us=c_{us}$.
The semigroup $T$ can be viewed concretely as follows: for $s\in S$, let $\rho_s\::\: S^0\rightarrow S^0$ be the right translation by $s$, and let $\gamma_s:S^0\rightarrow S$ be the constant mapping with value $s$.
Then $T$ is isomorphic to the subsemigroup $\{\rho_s,\gamma_s\::\: s\in S\}$ of $T_{S^0}$. 
It is clear that $T$ is orthodox and that $I$ is a (completely simple) minimal ideal with infinitely many (trivial) $\el$-classes.
Since the action of $S$ on $I$ is pseudo-finite, it follows from Proposition \ref{prop:rightideals} that $T$ is right pseudo-finite.
Of course, we can extend $S$ by left constants, by embedding it into the subsemigroup 
$\{\lambda_s,\gamma_s\::\: s\in S\}$ of the dual transformation monoid $T_{S^0}^\ast$, where
$\lambda_s$ is the left translation by $S$.
Finally we may extend $S$ by the rectangular band of left and right constants, 
by embedding into the subsemigroup $\{(\lambda_s,\rho_s),(\gamma_s,\gamma_t)\::\: s,t\in S\}$ of the direct product $T_{S^0}^\ast\times T_{S^0}$; this last monoid is orthodox with infinitely many $\ar$- and $\el$-classes in the minimal ideal, and is both left and right pseudo-finite.
%
\end{Ex}

We now turn our attention to the class of $\jay$-trivial monoids.  In what follows, a \emph{local zero} of an element $a$ in a semigroup $S$ is any idempotent $e\in E(S)$ such that $ae=ea=e.$

\begin{Lem}
Let $S$ be a $\jay$-trivial monoid and let $a\in S.$  An idempotent $e\in E(S)$ is a local zero of $a$ if and only if $e\,\leqj\,a.$
\end{Lem}

\begin{proof}
The forward implication is clear.  For the converse, we have that $e=sat$ for some $s, t\in S.$  Then $e=e^2=esat.$  Thus $e\,\ar\,es,$ so $e=es$ since $\ar\subseteq\jay$ and $S$ is $\jay$-trivial.  Then $e=eat,$ so $e\,\ar\,ea$ and hence $e=ea.$  A dual argument proves that $e=ae.$
\end{proof}

\begin{Thm}
\label{thm:J-trivial,lz}
Let $S$ be a $\jay$-trivial monoid in which every element has a local zero.  Then $S$ is right pseudo-finite if and only if it has a zero.
\end{Thm}

\begin{proof}
The reverse implication follows immediately from Corollary \ref{cor:rightideal}.
For the direct implication, let $D_r(S)=n,$ and let $X\subseteq S$ be a finite generating set for $\omega_S$ such that $D_r(X, S)=n.$  For each $u\in S,$ choose a local zero $u^{\ast}$ of $u.$  Consider $a\in S.$  There exists an $X$-sequence
$$1=x_1s_1,\ y_1s_1=x_2s_2,\ \cdots,\ y_ks_k=a$$
where $k\leq n.$  Then $x_1, s_1\in J_1,$ so $x_1=s_1=1$ since $S$ is $\jay$-trivial.
Let $e_1=y_1^{\ast}.$  Then $e_1\,\leqj\,y_1=y_1s_1.$
Thus, if $k=1$ then $e_1\,\leqj\,a.$  Suppose that $k>1.$  
For each $i\in\{1, \dots, k-1\},$ let $e_{i+1}=(y_{i+1}^{\ast}e_i)^{\ast}.$
Let $i\in\{1, \dots, k-1\}$ and assume that $e_i\,\leqj\,y_is_i.$  Then 
$$s_{i+1}\,\geqj\,x_{i+1}s_{i+1}=y_is_i\,\geqj\,e_i,$$ so $s_{i+1}e_i=e_i.$
Then
$$y_{i+1}s_{i+1}\,\geqj\,y_{i+1}^{\ast}y_{i+1}s_{i+1}e_i=y_{i+1}^{\ast}e_i\,\geqj\,(y_{i+1}^{\ast}e_i)^{\ast}=e_{i+1}.$$
Hence, by finite induction, we have that $e_k\,\leqj\,a.$  Now, the element $e_k$ depends only on the elements $y_1, \dots, y_k\in X.$  Since $X$ is finite and $k\leq n,$ it follows that there exists a finite set $V=\{v_1, \dots, v_m\}\subseteq E(S)$ with the following property: for any $a\in S$ there exists $i\in\{1, \dots, m\}$ such that $v_i\,\leqj\,a.$  Setting $z=v_1\dots v_m,$ we have that $z\,\leqj\,v_i$ for each $i\in\{1, \dots, m\},$ and hence $z\,\leqj\,a$ for all $a\in S.$  Thus $z$ is a zero of $S.$
\end{proof}

A semigroup $S$ is said to be {\em periodic} if every monogenic subsemigroup of $S$ is finite (equivalently, for every $a\in S$ there exist $q, r\in\N$ such that $a^{q+r}=a^q$).

\begin{Cor}
\label{cor:periodic,J-trivial}
A periodic $\jay$-trivial monoid $S$ is right pseudo-finite if and only if it has a zero.
\end{Cor}

\begin{proof}
Let $a\in S$ be arbitrary.  Since $S$ is periodic, there exist $q, r\in\N$ such that $a^{q+r}=a^q.$  Then $a^q\,\jay\,a^{q+1},$ so $a^q=a^{q+1}$ as $S$ is $\jay$-trivial.  It follows that $a^q$ is a local zero of $a.$  The result now follows from Theorem \ref{thm:J-trivial,lz}.
\end{proof}

In the remainder of this section we consider a wide class of monoids that includes all commutative monoids, namely right reversible monoids.

\begin{Def}
A monoid $S$ is \emph{right reversible} if for any $a, b\in S$ there exist $u, v\in S$ such that $ua=vb$.
\end{Def}

In addition to commutative monoids, the class of right reversible monoids contains all inverse monoids, monoids in which the $\mathcal{L}$-classes form a chain, monoids with a left zero, and left orders in groups.

Before we proceed, we provide the following curious result, giving a sufficient condition for right reversibility in terms of a generating set for the full right congruence.

\begin{Prop}\label{prop:rr} 
Let $T$ be a right reversible submonoid of a monoid $S.$
If $\omega_S$ has a generating set $X\subseteq T,$ then $S$ is right reversible.  
\end{Prop}

\begin{proof}  
Let $a, b\in S.$  Then there exists an $X$-sequence
$$a=x_1s_1,\ y_1s_1=x_2s_2,\ \cdots,\ y_ks_k=b$$
where $k\in\mathbb{N}.$  We prove by induction on $k$ that there exist $u, v\in T$ such that $ua=vb.$  Suppose first that $k=1.$  Since $x_1, y_1\in T$ and $T$ is right reversible, there exist $u, v\in T$ such that $ux_1=vy_1.$  Then $ua=vb.$  Now let $k>1$ and assume that there exist $w, z\in T$ such that $wa=z(x_ks_k).$  Since $zx_k, y_k\in T$ and $T$ is right reversible, there exist $u, v\in T$ such that $uzx_k=vy_k$ and hence $(uw)a=vb.$ 
\end{proof}

If $\omega_S$ is generated by a single pair of the form $(1, s),$ the assumptions of Proposition \ref{prop:rr} are satisfied with $T$ the submonoid generated by $s$, and hence $S$ is right reversible.

We now make a couple of technical definitions.

\begin{Def} Let $S$ be a monoid.  We say that a subset $V\subseteq S$ is an {\em absorbing set} for $S$ if for any $a\in S$ there exist $u, v\in V$ such that $ua=v.$
We say that $S$ is {\em finitely absorbed} if it has a {\em finite} absorbing set.
\end{Def}

\begin{Def}  We say that a monoid $S$ has {\em special right radius 2} if there exists a finite set $X\subseteq S$ such that for any $a\in S$ we have an $X$-sequence
\[a=x_1s_1,\, y_1s_1=x_2s_2,\, y_2s_2=1\]
where $x_1$ is right invertible.
\end{Def}

It is clear that any monoid with special right radius 2 is right pseudo-finite.
We remark that monoids with right diameter 1 and monoids with zero have special right radius 2.



We will later see how pseudo-finiteness interacts with right reversibility and the property of being finitely absorbed.  First, we show that the latter condition itself implies that the monoid is right pseudo-finite (indeed, it satisfies the stronger property of having special right radius 2). 

\begin{Def}
\label{def:wrr}
A monoid $S$ is \emph{weakly right reversible} if for any infinite sequence $$Sa_1, Sa_2,\dots$$ of principal left ideals of $S,$ there exist $i, j\in\mathbb{N}$ with $i<j$ such that $Sa_i\cap Sa_j\neq \emptyset$.
\end{Def}

\begin{Prop}\label{prop:condition(A)}
The following are equivalent for a monoid $S$:
\begin{thmenumerate}
\item $S$ is finitely absorbed;
\item $S$ has a completely simple minimal ideal with finite $\mathcal{R}$-classes;
\item $S$ is weakly right reversible and has special right radius 2.
\end{thmenumerate}
\end{Prop}

\begin{proof} 
(1)$\Rightarrow$(2).  Let $V$ be a finite absorbing set of $S.$  Observe that, for $u, v\in V,$ if $ua=v$ then $v\leq_{\mathcal{L}}a.$  Since the set $V$ is finite, we deduce that $S$ has minimal left ideals; the union of these minimal left ideals is the minimal ideal of $S,$ say $K.$  Choose $w\in S$ such that $L_w$ is a minimal left ideal of $S.$
Now, we have that $w\,\el\,w^i$ for all $i\in \N.$ Since $V$ is absorbing,  for each $i\in \N$ there exists $(u_i, v_i)\in V$ such that $u_iw^i=v_i.$  But since $V$ is finite, we must have
$(u_i, v_i)=(u_{i+j}, v_{i+j})$ for some $i, j\geq 1$.  Letting $u=u_i=u_{i+j}$ and $v=v_i=v_{i+j},$ we have that
\[vw^j=uw^iw^j=uw^{i+j}=v.\]
Since $v\leq_{\mathcal{L}}w,$ by the choice of $w$ we have that $w\,\el\,v,$ so that $w=tv$ for some $t\in S$.  It follows that $w=tvw^j=w^{j+1},$ so $w$ is periodic and hence $K$ contains an idempotent.  Thus $K$ is completely simple.  For any $a\in K,$ there exist $u, v\in V$ such that $ua=v,$ so $(wu)a=wv.$  It follows that every element of $K$ is $\el$-related to some $wv,$ where $(u, v)\in V$ for some $u\in S.$  Hence, $K$ has finitely many $\el$-classes.  Now consider a maximal subgroup $G=H_e$ of $K.$  For any $g\in G,$ there exist $u, v\in V$ such that $ug=v.$  Now, $eue,eve\in G$ and so  $g=(eue)^{-1}eve.$  Thus $G=\{(eue)^{-1}eve : u, v\in V\}.$  Since $V$ is finite, we conclude that $G$ is finite.  It follows that the $\ar$-classes of $K$ are finite.

(2)$\Rightarrow$(3).  Certainly $S$ is weakly right reversible.  Let $R_e$ be an $\ar$-class in the completely simple minimal ideal, and let $X$ denote the finite set $\{1\}\cup R_e$.  For any $a, b\in S,$ we have an $X$-sequence
$$a=1a,\, ea=(ea)1,\, (1)1=1.$$
Thus $S$ has special right radius 2.

(3)$\Rightarrow$(1).  
By assumption $S$ has special right radius $2$, so let $X\subseteq S$ be a finite set witnessing this property.
Consider an arbitrary $a\in S$.
Then there is
an $X$-sequence
\begin{equation}
\label{eq:ssr2}
a=x_1s_1,\, y_1s_1=x_2s_2,\, y_2s_2=1
\end{equation}
where $x_1$ is right invertible, say with $x_1t=1$.  
Also, clearly, $y_2$ is right invertible, with right inverse $s_2$.
Considering the sequence of principal right ideals $Sx_2y_2,Sx_2y_2^2,\dots$ and using the fact that $S$ is weakly right reversible, there exist $u, v\in S$ and $m, n\in\mathbb{N}$ such that
$$ux_2y_2^{m+n}=vx_2y_2^m.$$
Similarly, having chosen $v$, there exist $p, q\in S$ and $i, j\in\mathbb{N}$ such that
$$pvy_1x_1^{i+j}=qvy_1x_1^i.$$
Then
\begin{align*}
(pvy_1x_1^{j-1})a&=pvy_1x_1^js_1=pvy_1x_1^{i+j}t^is_1=qvy_1x_1^it^is_1=qvy_1s_1=qvx_2s_2\\
&=qvx_2y_2^ms_2^{m+1}=qux_2y_2^{m+n}s_2^{m+1}=qux_2y_2^{n-1}.
\end{align*}

Hence the set of all elements $pvy_1x_1^{j-1}$, $qux_2y_2^{n-1}$, as $a$ runs through $S$, is an absorbing set for $S$.
In fact this set can be chosen to be finite, by noticing that the choices of individual factors $u$, $v$, $p$, $q$ and exponents $j$ and $n$, featuring in the above elements, which we have made along the way, depend only
upon the sequence $(x_1,y_1,x_2,y_2)$ appearing in \eqref{eq:ssr2}.
There are only finitely many such sequences because $X$ is finite.
Therefore $S$ is finitely absorbed, as required.
\end{proof}

The following result provides several equivalent characterisations for a monoid to be both right pseudo-finite and right reversible.

\begin{Thm}\label{thm:rr} 
The following are equivalent for a monoid $S$:
\begin{thmenumerate}  
\item $S$ is right pseudo-finite and right reversible; 
\item $S$ is finitely absorbed and right reversible;
\item $S$ has a minimal ideal of the form $L\times G$ where $L$ is a left zero semigroup and $G$ is a finite group;
\item $S$ is right pseudo-finite and has a single minimal left ideal.
\end{thmenumerate}
\end{Thm}

\begin{proof}
(1)$\Rightarrow$(2).  Suppose that $\omega_S$ is generated by a finite set $X\subseteq S,$ and let $n=D(X, S).$
Let $\mathcal{S}$ be the set of sequences of elements of $X$ of even length no greater than $2n.$  Consider a sequence \[(x_1, y_1, x_2, y_2, \hdots, x_k, y_k)\] in $\mathcal{S}.$  Since $S$ is right reversible, there exist $u_1, v_1,\hdots, u_k, v_k\in S$ such that 
\[u_1x_1=v_1y_1,\, u_2v_1x_2=v_2y_2,\hdots, \, u_kv_{k-1}x_k=v_ky_k.\]
Let \[V=\{u_k\hdots u_2u_1, v_k: (x_1, y_1, \hdots, x_k, y_k)\in \mathcal{S}\};\]
notice that $V$ is finite.  We claim that $V$ is an absorbing set of $S$.  Indeed, let $a\in S.$  Then there exists an $X$-sequence
\[a=x_1s_1, \, y_1s_1=x_2s_2,\hdots, y_ks_k=1\]
where $k\leq n.$  Let the elements $u_1, v_1, \hdots, u_k, v_k$ be chosen as above. 
Then \[u_1a=u_1x_1s_1=v_1y_1s_1.\]
Suppose for induction that \[u_i\hdots u_2u_1a=v_iy_is_i.\]
Then 
\[u_{i+1}u_i\hdots u_2u_1a=u_{i+1}v_iy_is_i=u_{i+1}v_ix_{i+1}s_{i+1}=v_{i+1}y_{i+1}s_{i+1},\]
so that in a finite number of steps
\[u_k\hdots u_2u_1a=v_ky_ks_k=v_k,\]
as required.\par
(2)$\Rightarrow$(3).  By Proposition \ref{prop:condition(A)}, $S$ has a completely simple minimal ideal $K$ with finite $\mathcal{R}$-classes.  Let $w\in K$ be arbitrary.  By right reversibility, for any $p\in K$ the exist $x, y\in S$ such that $xp=yw.$  But then $p\,\el\,xp=yw\,\el\,w.$  Thus $K=L_w.$  We conclude that $K\cong L\times G$ where $L$ is a left zero semigroup and $G$ is a finite group.\par
(3)$\Rightarrow$(4).  For any $e\in L,$ the set $\{e\}\times G\cong G$ is a finite right ideal of $S.$  Therefore, by Proposition \ref{prop:rightideals}, $S$ is right pseudo-finite.  It is obvious that $L\times G$ is a minimal left ideal.\par
(4)$\Rightarrow$(1).  Let $L$ be the minimal left ideal of $S.$  Then $L$ is the minimal ideal of $S.$  Let $a, b\in S.$  Picking any $u\in L,$ we have that $ua, ub\in L.$  Since $ua\,\el\,ub,$ there exists $v\in L$ such that $vua=ub.$  Thus $S$ is right reversible.
\end{proof}

\begin{Cor}
Let $M$ be a commutative monoid.  Then $M$ is right pseudo-finite if and only it it has a minimal ideal that is a finite group.
\end{Cor}

Proposition \ref{prop:inv} is also an immediate corollary of Theorem \ref{thm:rr}.

%
%

%

\section{Minimal Ideals}\label{sec:mi}

In this section we show that if either of the partial orders $\leqel$ or $\leqj$ is left compatible with multiplication in a right pseudo-finite monoid $S$ then $S$ must have a minimal ideal.  Each of these two results is derived as a corollary of more technical necessary and sufficient conditions for a right pseudo-finite monoid to have a minimal left ideal or a minimal two-sided ideal.
 For clarity, in what follows the relation $\leqel$ or $\leqj$ always denotes the relation $\leqel$ or $\leq_{\mathcal{J}}$ in the parent semigroup.  The relation $\leqel$ is always right compatible with multiplication, but need not be left compatible; the relation $\leqj$ need not be either.

%

\begin{Thm}\label{thm:minleftideal} 
Let $S$ be a right pseudo-finite monoid.  The following are equivalent:
\begin{thmenumerate}
\item $S$ has a minimal left ideal;
\item $S$ has a left ideal $I$ such that $\leqel\cap(I\times I)$
is left compatible with multiplication in $S$;
\item for each $a\in S$ there exists $k=k_a\in S$ such that for any $u, v\in S,$
\[u\,\leqel\,v \mbox{ implies that }auk\,\leqel\,avk;\]
\item $\omega_S$ has a finite generating set $X\subseteq S,$ and for each $x\in X$ there exists $k=k_x\in S$ such that for any $u, v\in S,$
 \[u\,\leqel\,v \mbox{ implies that }xuk\,\leqel\,xvk.\]
\end{thmenumerate}
\end{Thm}

\begin{proof} (1)$\Rightarrow$(2)
Choosing a minimal left ideal $Sk$ of $S$, it is immediate from the minimality of $Sk$ that for any
$a\in S$ and $u,v\in S$ we have $auk\,\el\, uk\,\el\, k\, \el\, vk\,\el\ avk$, whence (2) trivially holds.

(2)$\Rightarrow$(3) Let $I$ be a left ideal satisfying the conditions of (2). Let $k\in I$ and $a\in S$. For any $u,v\in S$ with $u\,\leqel\, v$ we have that $uk\,\leqel\, vk$, since $\leqel$ is right compatible, and clearly $uk,vk\in I$. By assumption, $auk\,\leqel\, avk$. Thus (3) holds with 
 $k_a=k$ for all $a\in S.$ 
 
(3)$\Rightarrow$(4)  This is immediate.
 
(4)$\Rightarrow$(1) Suppose  that (4) holds.  Let $N=D(X, S).$  It is convenient to assume that $1\in X$ (we can take $k_1=1$).  Let $a\in S.$  There exists an $X$-sequence
\[a=x_1t_1,\, y_1t_1=x_2t_2, \dots,\, y_{n-1}t_{n-1}=x_nt_n,\, y_nt_n=1,\]
where $n\leq N.$  Let $s\in S.$  Then we have
\[as=x_1s_1,\, y_1s_1=x_2s_2, \dots,\, y_{n-1}s_{n-1}=x_ns_n,\, y_ns_n=s,\]
where $s_i=t_is$ for $1\leq i\leq n.$  Certainly $s\,\leqel\,s_n,$ so that $x_nsk_{x_n}\,\leqel\,x_ns_nk_{x_n}.$  Suppose for finite induction that for some $1<j\leq n$ we have 
\[x_jx_{j+1}\dots x_nsk_{x_n}\dots k_{x_{j+1}}k_{x_j}\,\leqel\,x_js_jk_{x_n}\dots k_{x_{j+1}}k_{x_j}.\]
Then, since $x_js_j=y_{j-1}s_{j-1},$ we have
\[x_jx_{j+1}\dots x_nsk_{x_n}\dots k_{x_{j+1}}k_{x_j}\,\leqel\,s_{j-1}k_{x_n}\dots k_{x_{j+1}}k_{x_j}.\]
It follows that 
\[x_{j-1}x_jx_{j+1}\dots x_nsk_{x_n}\dots k_{x_{j+1}}k_{x_j}k_{x_{j-1}}\,\leqel\,x_{j-1}s_{j-1}k_{x_n}\dots k_{x_{j+1}}k_{x_j}k_{x_{j-1}}.\]
By finite induction we have 
\begin{equation}\label{eqtn:minleftideal}
x_1x_2\dots x_nsk_{x_n}\dots k_{x_{2}}k_{x_1}\,\leqel\,x_1s_1k_{x_n}\dots k_{x_2}k_{x_1}=ask_{x_n}\dots k_{x_2}k_{x_1}.
\end{equation}
Hence, there is a finite set $V\subseteq S\times S$ with the following property: for any $a\in S$ there exists $(u, v)\in V$ such that
\[usv\,\leqel\,asv\]
for all $s\in S.$  Enumerate the elements of $V$ as $(p_1, q_1), \dots, (p_m, q_m),$ and let $q=q_1\dots q_m.$  Let $k\in\{1, \dots, m\}$ be such that $p_kq$ is minimal under $\leqel$ amongst $\{p_jq : 1\leq j\leq m\}.$
We claim that $L=Sp_kq$ is a minimal left ideal of $S.$  Clearly $L$ is a left ideal, so it suffices to prove that $L$ is the $\el$-class of $p_kq.$  So, let $t$ be any element of $S$ and consider $tp_kq.$  There exists $(p_i, q_i)\in V$ such that for any $s\in S$ we have
\[p_isq_i\,\leqel\,(tp_k)sq_i.\]
Taking $s=q_1\dots q_{i-1},$ we have
\[p_iq_1\dots q_{i-1}q_i\,\leqel\,(tp_k)q_1\dots q_{i-1}q_i.\]
Then, since $\leqel$ is right compatible, multiplying on the right by $q_{i+1}\dots q_m$ we obtain 
\[p_iq\,\leqel\,tp_kq.\]
Then $p_iq\,\leqel\,p_kq.$  Since $p_kq$ is minimal under $\leqel$ amongst $\{p_jq : 1\leq j\leq m\},$ we have that $p_iq\,\el\,p_kq.$  It follows that $tp_kq\,\el\,p_kq.$  This completes the proof. 
\end{proof}

Theorem~\ref{thm:minleftideal} applies to any monoids such that $uS\subseteq Su$ for any $u\in S$.

\begin{Cor}\label{cor:Lcompatible}
Let $S$ be a right pseudo-finite monoid.  If $\leqel$ is left compatible with multiplication, then $S$ has a minimal ideal that is the union of finitely many minimal left ideals.
\end{Cor}

\begin{proof}
Let $X$ be as given in the statement of Theorem \ref{thm:minleftideal}.  Since $\leqel$ is left compatible by assumption, in (4) of the statement of Theorem \ref{thm:minleftideal} we can set $k_x=1$ for all $x\in S.$  Let 
$$Q=\{x_1\dots x_n : x_i\in X, n\leq N\}.$$  Clearly $Q$ is finite.  From the equation (\ref{eqtn:minleftideal}) in the proof of Theorem \ref{thm:minleftideal} with $s=1,$ we see that for any $a\in S$ there exists some $x=x_1\dots x_n\in Q$ such that $x\leq_{\mathcal{L}}a.$  It follows that $S$ has finitely many minimal left ideals (the union of which is the minimal ideal of $S$).
\end{proof}

Corollary~\ref{cor:Lcompatible} yields another proof that a  commutative pseudo-finite monoid must have a minimal (left) ideal.


We now consider the existence of a minimal ideal in a right pseudo-finite monoid.

\begin{Thm}\label{thm:minideal} 
Let $S$ be a right pseudo-finite monoid.  The following are equivalent:
\begin{thmenumerate}
\item $S$ has a minimal ideal;
\item $S$ has an ideal $I$ such that $\leqj\cap (I\times I)$
is left compatible with multiplication in $I$;
\item for each $a\in S$ there exists $k=k_a\in S$ such that for any $u, v\in S,$
\[u\,\leqj\,v \mbox{ implies that }auk\,\leqj\,avk;\]
\item $\omega_S$ has a finite generating set $X\subseteq S,$ and for each $x\in X$ there exists $k=k_x\in S$ such that for any $u, v\in S,$
 \[u\,\leqj\,v \mbox{ implies that }xuk\,\leqj\,xvk.\]
\end{thmenumerate}
\end{Thm}

\begin{proof} 
Suppose that (1) holds.  The argument that (2) holds is as in Theorem~\ref{thm:minleftideal}.  Further, letting $k$ be any element of the minimal ideal of $S,$ it is clear that (3) holds with $k_a=k$ for all $a\in S.$  Clearly (3) implies (4).

To see that (2) implies (4), let $I$ be the ideal witnessing (2). 
Letting $X$ be a finite generating set of $\omega_S$. Replacing $X$ with $\{ 1,w \}\cup wX$ if necessary, where $w\in I$, we can assume that $X=\{ 1\} \cup Y$, where $Y\subseteq I$. 
It is clear that we can put $k_1=1$.  Now let $y\in Y$ and fix $k\in I$.  Let $u,v\in S$ with $u\,\leqj\, v$, so that $u=pvq$ for some $p,q\in S$.  Certainly $uk=pvqk$ and $uk, vqk\in I$ with $uk\,\leqj\, vqk$.  By assumption, $yuk\,\leqj\, yvqk$.  Further, $qk\,\leqj\, k$ and $qk,k\in I$, so that again by the compatibility assumption, $yvqk\,\leqj\, yvk$. By transitivity of
$\leqj$ we have $yuk\,\leqj\, yvk$ and so (4) holds.

Suppose now that (4) holds.  By essentially the same argument as the one in the proof of (4)$\Rightarrow$(1) of Theorem \ref{thm:minleftideal}, with $s=1$ and $\leqj$ instead of $\leqel$, we obtain a finite set $V\subseteq S$ with the following property: for any $a\in S$ there exists $v\in V$ such that $v\,\leqj\,a.$  Enumerating the elements of $V$ as $q_1, \dots, q_m$ and letting $q=q_1\dots q_m,$ it follows that $q\,\leqj\,a$ for all $a\in S.$  Thus $J_q$ is the minimal ideal of $S.$
\end{proof}

\begin{Cor}\label{cor:Jcompatible}
Let $S$ be a right pseudo-finite monoid.  If $\leqj$ is left compatible  with multiplication, then $S$ has a minimal ideal.
\end{Cor}


To see how Corollary~\ref{cor:Jcompatible} may be applied, let
$S$ be a semilattice $Y$ of semigroups $S_{\alpha}$ where the $S_{\alpha}$ are simple.  It is easy to see that the $S_{\alpha}$ are the $\jay$-classes of $S$, and that $a\,\leqj\, b$ if and only if $a\in S_{\alpha},b\in S_{\beta}$ and $\alpha\leq \beta$.
(Here we are using the natural order in a semilattice given by $x\leq y$ if and only if $x=xy=yx$.)  Since the order in the semilattice is compatible, it is immediate that $\leqj$ is compatible in $S$.  This yields another proof that right pseudo-finite completely regular semigroups (including bands) must have minimal ideals. 





\section{A Rees Matrix Semigroup Extension}\label{sec:con}

In this section we introduce a construction, in the form of a specific ideal extension of a Rees matrix semigroup, which is then used to exhibit a series of examples illustrating our findings from the previous sections, and to explore their limitations.
We first introduce the construction in its most general form, and then give a special instance of it that will be particularly useful in constructing various examples without minimal ideals in the next section.

%
%
%

\begin{con}\label{con1}
Let $S$ and $T$ be semigroups.  Let $I$ be a left $S$-act and let $J$ be a right $S$-act.  Let $P=(p_{j,i})$ be a $J\times I$ matrix with entries from $T$ such that $p_{js,i}=p_{j,si}$ for all $i\in I,$ $j\in J$ and $s\in S.$  Let $M=S^1\cup\mathcal{M}[T; I, J; P].$  
Define a multiplication on $M,$ extending those on $S^1$ and $\mathcal{M}[T; I, J; P],$ as follows:
$$s(i, t, j)=(si, t, j)\text{ and }(i, t, j)s=(i, t, js)$$
for all $i\in I,$ $j\in J,$ $s\in S^1$ and $t\in T.$  One can check by an exhaustive case analysis that this multiplication is associative, and hence $M$ is a monoid with identity 1.  Here is a sample case, in which the assumption $p_{js,i}=p_{j,si}$ is used:
\begin{align*}
\bigl((i, t, j)s\bigr)(k, u, l)&=(i, t, js)(k, u, l)=(i, tp_{js,k}u, l)=(i, tp_{j,sk}u, l)
=(i, t, j)(sk, u, l)\\&=(i, t, j)\bigl(s(k, u, l)\bigr).
\end{align*}
We denote the monoid $M$ by $\mathcal{E}(S, T; I, J; P).$ 
We permit $S$ to be empty in this construction, in which case $\mathcal{E}(S, T; I, J; P)$ is simply $\mathcal{M}[T; I, J; P]^1$.
\end{con}

\begin{Rem}
\label{re-byleen}
The above construction is closely related to another matrix construction, introduced by Byleen \cite{Byleen:1988}.
Specifically, the monoid $\mathcal{E}(S, T; I, J; P)$ can be found as a subsemigroup inside some Byleen's monoid $\mathscr{C}(U;\beta,\alpha;P)^1,$ where $U$, $\alpha$, $\beta$ are as outlined as follows.
Let $U=S\cup T,$ and define a multiplication on $U,$ extending those on $S$ and $T,$ by $st=ts=s$ for all $s\in S$ and $t\in T$.
The left (resp.\ right) action of $S$ on $I$ (resp.\ $J$) can be extended to a left (resp.\ right) action $\beta$ (resp.\ $\alpha$) of  $U$ on $I$ (resp.\ $J$) by setting $ti=i$, $jt=j$ for all $i\in I$, $j\in J$, $t\in T$.  It is then straightforward to check that all the conditions are satisfied for forming the Byleen semigroup $\mathscr{C}(U;\beta,\alpha;P),$ as specified in  \cite{Byleen:1988}.
Its elements are $I^+J^\ast\cup I^\ast UJ^\ast \cup I^\ast J^+$ \cite[Lemma 1.1]{Byleen:1988}.
The monoid $\mathcal{E}(S, T; I, J; P)$ embeds into $\mathscr{C}(U;\beta,\alpha;P)^1$ via $1\mapsto 1$,  $s\mapsto s$ ($s\in S$), $(i,t,j)\mapsto itj$ ($i\in I$, $j\in J$, $t\in T$).
This observation can be used to avoid checking associativity above.
\end{Rem}


It turns out that the construction yields a plethora of pseudo-finite monoids. The following result gives a sufficient condition for this to be the case.  It is interesting to compare it with Theorem \ref{thm:csmi}, Corollary \ref{Cor:cs1} and \cite[Theorem 6.5]{Dandan:2019}, which all deal with extensions of Rees matrix semigroups.

\begin{Thm}\label{thm:con1}
Let $M=\mathcal{E}(S, T; I, J; P).$  Suppose that the following conditions hold:
\begin{thmenumerate}
\item $T=T^1X$ for some finite set $X\subseteq T$;
\item there exists a finite subset $J_0\subseteq J$ such that $T=\{p_{j,i} : j\in J_0, i\in I\}$;
\item there exists $j_0\in J_0$ such that $p_{j_0,i}\in X$ for all $i\in I$;
\item $J$ is pseudo-finite as a right $S$-act.
\end{thmenumerate}
Then $M$ is right pseudo-finite.
\end{Thm}

\begin{proof}
Without loss of generality we may assume that $\omega_J$ is generated by $J_0$ (otherwise, given a finite generating set $J_1\subseteq J$ of $\omega_J$, let $J_0^{\prime}=J_0\cup J_1$; then $J_0^{\prime}$ can replace $J_0$ in conditions (2) and (3) of the statement of the result).  Let $d$ denote the $J_0$-diameter $D(J_0, J),$ which is finite since $J$ is pseudo-finite.\par
Fix $i_0\in I.$  We prove that the right ideal $K=\{i_0\}\times XT^1\times J$ of $M$ is pseudo-finite as a right $M$-act, from which it follows that $M$ is right pseudo-finite by Lemma \ref{lem:subact}.  Let $U=X\cup X^2\cup X^3,$ and let $$H=\{(i_0, u, j) : u\in U, j\in J_0\}.$$
Since $U$ and $J_0$ are finite, so is $H.$  We shall show that $\omega_K$ is generated by $H$ and that $D(H, K)\leq d+2.$  
Let $m=(i_0, xs, j), n=(i_0, yt, k)\in K,$ where $x, y\in X$ and $s, t\in T^1.$  We first claim that there exists some $u\in U$ with an $H$-sequence of length no greater than 1 from $m$ to $(i_0, u, j).$  If $s\in X\cup\{1\}$ then $m\in H,$ so we can just set $u=xs.$  Otherwise, by (1) we have $s=s^{\prime}z$ for some $s^{\prime}\in T$ and $z\in X.$  By (2) there exist $j^{\prime}\in J_0$ and $i\in I$ such that $s^{\prime}=p_{j^{\prime},i}.$  We have an $H$-sequence
$$m=(i_0, x, j^{\prime})(i, z, j),\; (i_0, x, j_0)(i, z, j)=(i_0, xp_{j_0, i}z, j).$$
Since $p_{j_0, i}\in X$ by (3), setting $u=xp_{j_0, i}z$ establishes the claim.  Similarly, there exists some $v\in U$ with an $H$-sequence of length no greater than 1 from $n$ to $(i_0, v, k).$  Now, there exists a $J_0$-sequence
$$j=j_1s_1, k_1s_1=j_2s_2, \dots, k_ls_l=k$$
where $l\leq d.$  Thus, we have an $H$-sequence
\begin{align*}
(i_0, u, j)=(i_0, u, j_1)s_1,\; (i_0, v, k_1)s_1=(i_0, v, j_2)s_2,\; (i_0, v, k_2)s_2=&(i_0, v, j_3)s_3, \dots,\\ 
&(i_0, v, k_l)s_l=(i_0, v, k).
\end{align*}
We conclude that there exists an $H$-sequence of length no greater than $d+2$ from $m$ to $n,$ as required.
\end{proof}

%

\begin{Rem}
\label{rem:orthrev}
The final example invoked in Example \ref{ex:orthodox} can be explicitly realised as $\mathcal{E}(S, T; I, J; P)$, where $T$ is trivial, the index sets are $I=\{i_s\::\: s\in S\}$ and $J=\{j_s\::\: s\in S\}$,
and the actions are given by $ti_s=i_{ts}$, $j_s t=j_{st}$ $(s, t\in S)$.
The right pseudo-finiteness follows from Theorem \ref{thm:con1}, and left pseudo-finiteness follows by duality.
\end{Rem}

\begin{Rem}
\label{rem:2minid}
In Theorem \ref{thm:rr} we have seen that if a right pseudo-finite monoid $S$ has a unique minimal left ideal then this has to be completely simple.  We can use our construction to show that there exists a right pseudo-finite monoid with precisely two minimal left ideals, but the minimal ideal of which is not completely simple.
Indeed, take $S$ to be empty, and let $T$ be any left simple semigroup that is not completely simple.
For the index sets, take $I=\{i_t\::\: t\in T\}$ and $J=\{1,2\}$, and pick the entries of $P$ to ensure one row is constant and the other contains all elements of $T.$
The semigroup $\mathcal{E}(S, T; I, J; P)=\mathcal{M}[T;I,J;P]^1$ has a minimal two-sided ideal $\mathcal{M}[T;I,J;P]$, which is a disjoint union of two minimal left ideals $I\times T\times\{j\}$ ($j=1,2$) and is not completely simple.
\end{Rem}

\begin{con}\label{con2}
Let $S$ be a semigroup such that $S=YS^1=S^1Y$ for some finite set $Y\subseteq S$ (that is, $S$ is finitely generated both as a right ideal and a left ideal).  Let $I=\{i_s : s\in S\}\cup\{0\},$ and define right and left actions of $S$ on $I$ as follows:
$$i_st=i_{st},\, ti_s=i_{ts},\, 0t=t0=0\quad (s,t\in S).$$
Fix $x\in Y,$ and let $P=(p_{i,j})$ be the $I\times I$ matrix with entries given by $p_{0,i}=p_{i,0}=x$ for all $i\in I$ and $p_{i_s,i_t}=st$ for all $s, t\in S.$  It is easy to see that $p_{is,j}=p_{i,sj}$ for all $i, j\in I$ and $s\in S.$  We denote the monoid $\mathcal{E}(S, S; I, I; P)$ by $\mathcal{E}(S, x)$.
In the case that $S$ is a monoid, we abbreviate $\mathcal{E}(S, 1_S)$ to $\mathcal{E}(S).$
\end{con}

\begin{Cor}\label{cor:con2pf}
For any $S$ and $x$ as in Construction \ref{con2}, the monoid $\mathcal{E}(S, x)$ is both right and left pseudo-finite.
\end{Cor}

\begin{proof}
Let $M=\mathcal{E}(S, x).$  We prove that $M$ is right pseudo-finite; the proof of left pseudo-finiteness is dual.

Let $T=\{x\}\cup S^2.$  Clearly $T$ is an ideal of $S.$  Notice that the entries of $P$ are precisely the elements of $T.$  Recalling Construction \ref{con1}, let $K=\mathcal{E}(S, T; I, I; P).$  It is easy to see that $K$ is an ideal of $M.$  Therefore, by Corollary \ref{cor:rightideal}, it suffices to prove that $K$ is right pseudo-finite.  We show that $K$ satisfies the conditions of Theorem \ref{thm:con1}.  Let $X=\{x\}\cup Y^2\cup Y^3\subseteq T$ where $Y$ is as given in Construction \ref{con2}.  Clearly $X$ is finite since $Y$ is finite.  Using the fact that $S=S^1Y,$ we have that
$$S=Y\cup SY=Y\cup S^1Y^2=Y\cup Y^2\cup SY^2.$$  It follows that $S^2\subseteq Y^2\cup Y^3\cup S^2Y^2\subseteq T^1X.$  Thus $T=\{x\}\cup S^2\subseteq T^1X,$ and hence $T=T^1X.$  Thus condition (1) holds.\par
Now consider $t\in T.$  Then $t=x$ or $t\in S^2.$  In the former case, we have $t=p_{0,i}$ for any $i\in I.$  In the latter case, since $S=YS^1$ it follows that $t=ys$ for some $s\in S,$ and hence $t=p_{i_y,i_s}.$  Thus condition (2) holds with $J_0=\{i_y : y\in X\}\cup\{0\}.$  Clearly (3) holds with $j_0=0.$  Finally, observe that the $S$-act $I$ is finitely generated by $J_0$ and contains the trivial subact $\{0\},$ which is certainly pseudo-finite, so $I$ is pseudo-finite by Lemma \ref{lem:subact}.  Hence, by Theorem \ref{thm:con1}, $M$ is right pseudo-finite.
\end{proof}

\section{No Minimal Ideal}
\label{sec:nomids}

In this section we discuss pseudo-finite semigroups without minimal ideals.
We have already seen that it is possible for a right pseudo-finite semigroup to have a minimal ideal that is not completely simple:
the Baer-Levi semigroup $\mathcal{BL}_X$ and the monoid $\mathcal{F}_X$ both have this property, as discussed in Section \ref{sec:pf,ideals}, as does the monoid constructed in Remark \ref{rem:2minid}.

Our first example of a pseudo-finite monoid with no minimal ideal will be another transformation monoid.

\begin{Ex}
Let $X$ be an infinite set, let $\{X_i : i\in\mathbb{N}\}$ be a partition of $X$ where $|X_i|=|X|$ for each $i\in\mathbb{N}
,$ and define
\[\mathcal{U}_X=\Bigl\{\alpha\in\mathcal{T}_X: X_i\alpha\subseteq \bigcup_{j\geq i}X_j\text{ for each }i\in\mathbb{N}\Bigr\}.\]
It can be easily shown that $\mathcal{U}_X$ is a submonoid of $\mathcal{T}_X.$
We claim that both the diagonal right act and diagonal left act of $\mathcal{U}_X$ are cyclic, so that $\mathcal{U}_X$ is both right and left pseudo-finite, but that $\mathcal{U}_X$ has no minimal ideal.

We first prove that the diagonal right act of $\mathcal{U}_X$ is cyclic.  For each $i\geq\mathbb{N},$ let $\{X_{i\alpha}, X_{i\beta}\}$ be a partition of $X_i$ into two sets with cardinality $|X|,$ and let \[\alpha_i:X_i\rightarrow X_{i\alpha},\, \beta_i:X_i\rightarrow X_{i\beta}\] be bijections.  Put
\[\alpha=\bigcup_{i\in\mathbb{N}}\alpha_i,\; \beta=\bigcup_{i\in\mathbb{N}}\beta_i,\]
so that $\alpha,\beta\in\mathcal{U}_X.$  We claim that $(\alpha, \beta)$ generates the diagonal right act of $\mathcal{U}_X.$  Indeed, let $(\gamma,\delta)\in\mathcal{U}_X.$  Define $\theta\in\mathcal{T}_X$ by
\[x\theta=\begin{cases}
x\alpha_i^{-1}\gamma&\text{ if }x\in X_{i\alpha}\\
x\beta_i^{-1}\delta&\text{ if }x\in X_{i\beta}.
\end{cases}\]
It is clear that $\theta\in\mathcal{U}_X$ and that $(\gamma, \delta)=(\alpha, \beta)\theta,$ as required.

We now consider the diagonal left act of $\mathcal{U}_X.$  For each $i\in I,$ let 
$$
\phi_i : X_i\to
\bigcup_{j\geq i}\bigl((X_i\times X_j)\cup (X_j\times X_i)\bigr)
$$
be a bijection, and put 
$$\phi=\bigcup_{i\in\mathbb{N}}\phi_i : X\to X\times X.$$  
It is straightforward to show that $\phi$ is a bijection.  Set $\alpha=\phi p_1$ and $\beta=\phi p_2,$ where $p_1, p_2$ are the projections onto the first and second coordinates, respectively.  We claim that $(\alpha, \beta)$ generates the diagonal left act of $\mathcal{U}_X.$  Indeed, let $(\gamma,\delta)\in\mathcal{U}_X.$  Define a map $$\theta : X\to X, x\mapsto(x\gamma, x\delta)\phi^{-1}.$$
Consider any $x\in X.$  Then $x\in X_i$ for some $i\in\mathbb{N},$ and hence $x\gamma\in X_j$ and $x\delta\in X_k$ for some $j, k\geq i.$  It follows from the definition of $\phi^{-1}$ that $x\theta\in X_m,$ where $m=\min(j, k).$  Thus $\theta\in\mathcal{U}_X.$  Furthermore, we have that 
$$x\gamma=(x\gamma, x\delta)p_1=(x\gamma, x\delta)(\phi^{-1}\phi)p_1=x(\theta\alpha),$$
so $\gamma=\theta\alpha.$  Similarly, $\delta=\theta\beta.$  Thus $(\gamma, \delta)=\theta(\alpha, \beta),$ as required.

Finally, suppose for a contradiction that $\mathcal{U}_X$ has a minimal ideal, and let $\alpha$ be any element of the minimal ideal.
Fix $x\in X_1$ and let $x\alpha\in X_k.$  
Choose $\beta\in\mathcal{U}_X$ such that $X\beta\subseteq\bigcup_{i\geq k+1}X_i.$
Now, there exist $\gamma, \delta\in\mathcal{U}_X$ such that $\alpha=\gamma\beta\delta.$
Since $(x\gamma)\beta\in\bigcup_{i\geq k+1}X_i,$ we have that $x\alpha=\bigl((x\gamma)\beta\bigr)\delta\in\bigcup_{i\geq k+1}X_i,$ contradicting the fact that $x\alpha\in X_k.$
\end{Ex}

A wealth of further examples can be obtained by utilising the constructions from the previous section, following this easy observation:

\begin{Prop}\label{prop:minideal}
For $S,$ $T,$ $I,$ $J$ and $P$ as in Construction \ref{con1}, the monoid $M=\mathcal{E}(S, T; I, J; P)$ has a minimal ideal if and only if $T$ has a minimal ideal.
\end{Prop}


\begin{proof}
Suppose that $M$ has a minimal ideal, say $K.$
Then $K$ must contain an element of the form $(i, a, j).$  Let $t$ be any element of $T.$  Then $(i, a, j)=m(i, t, j)n$ for some $m, n\in M.$  It follows that $a\in T^1tT^1,$ so $a\leq_{\mathcal{J}}t.$  Thus $J_a$ is the minimal ideal of $T.$

Conversely, suppose that $T$ has a minimal ideal, say $L.$  We claim that the ideal $K=I\times L\times J$ is minimal in $M.$  
Consider $(i_1, u, j_1), (i_2, v, j_2)\in K$, where $u,v\in L$, $i_1,i_2\in I$, $j_1,j_2\in J$.
Since $vp_{j_2,i_2}vp_{j_2,i_2}v\in L$, there exist $y, z\in T^1$ such that
$yvp_{j_2,i_2}vp_{j_2,i_2}vz=u$. Note that $yv, vz\in T$, and 
\[
(i_1,u,j_1)=(i_1,yv,j_2)(i_2,v,j_2)(i_2,vz,j_1).
\]
It follows that $(i_1,u,j_1)\ \leqj \ (i_2,v,j_2)$, and, by symmetry, $(i_2,v,j_2)\ \leqj\ (i_1,u,j_1)$.
Hence the ideal $K$ is a $\jay$-class as well, and hence it is a minimal ideal, as required.
\end{proof}

It now follows that if we plug any $T$ with no minimal ideal into Construction \ref{con1}, while respecting the conditions of Theorem \ref{thm:con1}, we will obtain a right pseudo-finite monoid with no minimal ideal.  Moreover, if we put $S$ with no minimal ideal into Construction \ref{con2}, while respecting the conditions stipulated there, we will obtain a monoid with no minimal ideal that is both right pseudo-finite and left pseudo-finite.  These observations allow us to exhibit examples of pseudo-finite monoids without minimal ideals satisfying various prescribed properties.  This will complement our findings from Sections \ref{sec:csmi} and \ref{sec:mi} and show their natural limitations.

In Section \ref{sec:csmi} we discussed a number of subclasses of the class of regular monoids -- inverse, completely regular and orthodox monoids -- and they all turned out to have (completely simple) minimal ideals when pseudo-finite.  One therefore may wonder whether this can be extended to all regular monoids, or at least all idempotent-generated regular monoids.  To answer these questions in the negative we will resort to the variant $\mathcal{E}(S)$ from Construction \ref{con2}.

\begin{Prop}
\label{prop:Ereg}
Let $S$ be a monoid.  Then the monoid $\mathcal{E}(S)$ is regular if and only if $S$ is regular, and $\mathcal{E}(S)$ is idempotent-generated if and only if $S$ is idempotent-generated.
\end{Prop}

\begin{proof}
Let $M=\mathcal{E}(S).$  
The direct part of the two statements follow from the fact that $S^1$ is a submonoid of $M$ with an ideal complement.\par
Now suppose that $S$ is regular.  Clearly any element of $S^1$ is regular in $M,$ so consider $(i, s, j)\in\mathcal{M}(S; I, I; P).$  There exists $t\in S$ such that $s=sts,$ and hence, recalling that $p_{j,0}=p_{0,i}=1_S$, we have
$$(i, s, j)=(i, s, j)(0, t, 0)(i, s, j),$$
completing the proof that $S$ is regular.\par
Now suppose that $S$ is idempotent-generated.  Then certainly $S^1$ is idempotent-generated.  Note that all $(i, 1_S, 0), (0, 1_S, i)$ ($i\in I$) are idempotents.  For any $(i, s, j)\in\mathcal{M}(S; I, I; P),$ we have that
$$(i, s, j)=(i, 1_S, 0)(0, 1_S, i_s)(i_{1_S}, 1_S, 0)(0, 1_S, j),$$
so that $(i, s, j)$ is a product of idempotents.  Thus $M$ is idempotent-generated.
\end{proof}

Taking $S$ to be any (idempotent-generated) regular monoid with no minimal ideal, Corollary \ref{cor:con2pf} and Propositions \ref{prop:minideal} and \ref{prop:Ereg} together yield:

\begin{Cor}
There exist (idempotent-generated) regular monoids that are both right pseudo-finite and left pseudo-finite but have no minimal ideal.
\end{Cor}


Given Corollary \ref{cor:periodic,J-trivial}, it is natural to ask whether either of the properties of being periodic or being $\jay$-trivial is sufficient on its own to guarantee the existence of a minimal ideal in a pseudo-finite semigroup.  Again, we apply Construction \ref{con2} to show that this is not the case.

\begin{Prop}
There exist $\jay$-trivial monoids that are both right pseudo-finite and left pseudo-finite but have no minimal ideal.
\end{Prop}

\begin{proof}
Let $S$ be a semigroup such that $S=XS^1=S^1X$ for some finite set $X\subseteq S,$ and $a\notin S^1aS\cup SaS^1$ for each $a\in S.$ (For example, we can take $S$ to be the free semigroup on a finite set $X.$)  Fix $x\in X$ and let $M=\mathcal{E}(S, x).$  By Corollary~\ref{cor:con2pf}, $M$ is both right pseudo-finite and left pseudo-finite.  Clearly $S$ is $\jay$-trivial and has no minimal ideal.  Thus $M$ has no minimal ideal by Proposition \ref{prop:minideal}.\par
We now show that $M$ is $\jay$-trivial.  Suppose for a contradiction that $M$ is not $\jay$-trivial.  Let $T=\mathcal{M}[S; I, I; P].$  It is clear that the restriction of the $\jay$-relation on $M$ to $S$ is the $\jay$-relation on $S,$ which is the equality relation since $S$ is $\jay$-trivial, and no elements of $S$ are $\jay$-related to elements of $T.$  Therefore, there must exist two distinct elements $(i, a, j), (k, b, l)\in T$ with $(i, a, j)\jay(k, b, l).$  Then there exist $u, v, u^{\prime}, v^{\prime}\in M$ such that
$$u(i, a, j)v=(k, b, l),\, u^{\prime}(k, b, l)v^{\prime}=(i, a, j).$$  There are two cases to consider.\par
(1) $u\in T$ or $v\in T.$  Assume without loss of generality that $u=(q, s, r)\in T.$  Then 
$$(k, b, l)=(q, s, r)(i, a, j)v=(q, sp_{ri}a, j)v.$$
If $v\in S,$ then $b=sp_{ri}a.$  If $v=(q^{\prime}, t, r^{\prime}),$ then
$(q, sp_{ri}a, j)v=(q, sp_{ri}ap_{jq^{\prime}}t, r^{\prime}),$ and hence $b=sp_{ri}ap_{jq^{\prime}}t.$  In either case we have that $b\leq_{\mathcal{J}}a,$ and $b\neq a$ since $a\notin S^1aS\cup SaS^1.$\par
Notice that $S^1(k, b, l)S^1\subseteq I\times\{b\}\times I.$  Therefore, since $a\neq b,$ we must have that $u^{\prime}\in T$ or $v^{\prime}\in T.$  Then, by the same argument as above, we have that $a\leq_{\mathcal{J}}b.$  But then $a\,\jay\,b,$ contradicting the fact that $S$ is $\jay$-trivial.\par
(2) $u, v\in S^1.$  In this case we have $u(i, a, j)v=(ui, a, jv),$ so $ui=k,$ $a=b$ and $jv=l.$  Since $(i, a, j)\neq(k, b, l),$ it follows that $i\neq k$ or $j\neq l.$  Assume without loss of generality that $i\neq k.$  Then $u\neq 1.$ Now, we must have that $u^{\prime}\in S^1,$ for otherwise, by the argument in case (1), we would have $a\neq b.$  Therefore, we have that $(i, a, j)=(u^{\prime}k, b, l)v^{\prime},$ whence $i=u^{\prime}k.$  Since $i\neq k,$ we conclude that $u^{\prime}\neq 1.$  We cannot have $i=0$ or $k=0,$ since this would imply that $i=k=0.$  Thus there exist $s, t\in S$ with $s\neq t$ such that $i=i_s$ and $k=i_t.$  It follows that $s=u^{\prime}t$ and $t=us.$  But then $s\,\el\,t$ and hence $s\,\jay\,t,$ contradicting the fact that $S$ is $\jay$-trivial.
\end{proof}

\begin{Prop}\label{prop:periodic}
For $S$ and $x$ as in Construction \ref{con2}, the monoid $M=\mathcal{E}(S, x)$ is periodic if and only if $S$ is periodic.
\end{Prop}

\begin{proof}
Clearly periodicity is closed under subsemigroups, so $S$ is periodic if $M$ is.

Suppose that $S$ is periodic.  Since $S$ is a subsemigroup of $M,$ every monogenic subsemigroup of $M$ generated by an element of $S$ is contained in $S,$ and is hence finite since $S$ is periodic.  So, consider $m=(i, s, j)$ where $i, j\in I$ and $s\in S.$  Since $S$ is periodic, there exist $q, r\in\N$ such that and $(sp_{ji})^{q+r}=(sp_{ji})^q.$  Thus, we have
$$(i, s, j)^{q+1+r}=(i, (sp_{ji})^{q+r}s, j)=(i, (sp_{ji})^qs, j)=(i, s, j)^{q+1}.$$
This completes the proof.
\end{proof}

Taking $S$ to be any periodic semigroup with no minimal ideal (such as a semilattice with no zero), Corollary \ref{cor:con2pf} and Propositions \ref{prop:minideal} and \ref{prop:periodic} together yield: 

\begin{Cor}
There exist periodic monoids that are both right pseudo-finite and left pseudo-finite but have no minimal ideal.
\end{Cor}

We conclude this section by showing that there exist weakly right reversible monoids that are right pseudo-finite but have no minimal ideal.

\begin{Prop}
Let $T$ be any right reversible monoid with no minimal ideal.  Let $I$ be a set such that $|I|=|T|,$ let $J=\{1, 2\},$ and let $P$ be a $J\times I$ matrix such that $p_{1,i}=1_T$ for all $i\in I$ and every element of $T$ appears in the second row.  Then the monoid $M=\mathcal{M}[T; I, J; P]^1$ is right pseudo-finite, weakly right reversible and has no minimal ideal.
\end{Prop}

\begin{proof}
Recall that $M=\mathcal{E}(\emptyset, T; I, J; P).$
Clearly the conditions of Theorem \ref{thm:con1} are satisfied, so $M$ is right pseudo-finite.  Since $T$ has no minimal ideal, $M$ has no minimal ideal by Proposition \ref{prop:minideal}.  We now show that $M$ is weakly right reversible.  Consider any infinite sequence $$Mu_1, Mu_2, \dots$$ of principal left ideals of $M.$  Since $J$ is finite, there exists $j\in J$ such that there is an infinite subsequence $Mu_{k_1}, Mu_{k_2}, \dots$ where the third co-ordinate of each $u_{k_p}$ is $j.$  For each $p\in\N,$ let $u_{k_p}=(i_p, t_p, j).$  Since $T$ is right reversible, for any $p, q\in\N,$ there exist $v, w\in T$ such that $vt_p=wt_q.$  Picking any $i\in I,$ we then have that 
$$(i, v, 1)u_{k_p}=(i, vt_p, j)=(i, wt_q, j)=(i, w, 1)u_{k_q},$$ so that $Mu_{k_p}\cap Mu_{k_q}\neq\emptyset.$  Thus $M$ is weakly right reversible.
\end{proof}

\section{Open Problems and Future Research}
\label{sec:conclusion}

We conclude this paper with some open problems and possible directions for future research.

The notion of diameter has been a useful tool in this paper and is deserving of a more systematic investigation.  As indicated in Section \ref{subsec:defn}, this will be the topic of a subsequent paper.

A natural open problem arising from the work of this paper is to completely describe the minimal ideals in (right) pseudo-finite semigroups.  To put it another way, which simple semigroups can be the minimal ideal of a pseudo-finite semigroup?  Certainly not all simple semigroups have this property; e.g.\ any simple monoid that is not pseudo-finite, such as the bicyclic monoid or any infinite group, by Corollary \ref{cor:monoidideal}.


As noted in Section \ref{sec:intro}, in \cite{Dandan:2019} several equivalent characterisations were given for a monoid $S$ to have a finitely generated universal left congruence, including $S$ satisfying the homological finiteness property of being type left-$FP_1$.  This raises the question as to whether the property of being pseudo-finite could similarly be described in homological terms.


\section*{Acknowledgements}
This work was supported by the Engineering and Physical Sciences Research Council [EP/V002953/1].


\vspace{1em}
\begin{thebibliography}{99}
\bibitem{Bulman:1989} 
S. Bulman-Fleming and K. McDowell.  Problem e3311, {\em Amer. Math. Monthly}, 96:155, 1989.  Solution appeared in {\em Amer. Math. Monthly}, 97:167, 1990.
\bibitem{Byleen:1988}
K. Byleen. Embedding any countable semigroup without idempotents in a {$2$}-generated simple semigroup without idempotents.  \textit{Glasgow Math. J.}, 30:121-128, 1988.
\bibitem{Clifford:1948}
A.H. Clifford.  Semigroups containing minimal ideals.  \textit{Amer. J. Math.}, 70:521-526, 1948.
\bibitem{Clifford:1967}
A.H. Clifford and G.B. Preston.  {\em The Algebraic Theory of Semigroups: Volume II}.   Amer. Math. Soc., 1967.
\bibitem{Dandan:2019} 
Y. Dandan, V. Gould, T. Quinn-Gregson and R.-E Zenab.  Semigroups with finitely generated universal left congruence.  {\em Monat. Math.}, 190:689-724, 2019. 
\bibitem{Gallagher:2006}
P. Gallagher.  On the finite and non-finite generation of diagonal acts.  {\em Comm. Algebra}, 34:3123-3137, 2006.
\bibitem{Gallagher:2005}
P. Gallagher and N. Ru\v{s}kuc.  Generation of diagonal acts of some semigroups of transformations and relations.  {\em Bull. Australian Math. Soc.}, 72:139-146, 2005.
\bibitem{Grillet:1995}
P. Grillet.  {\em Semigroups: An Introduction to the Structure Theory}.  CRC Press, 1995.
\bibitem{Howie:1995} 
J.M. Howie.  {\em Fundamentals of Semigroup Theory}.  OUP, Oxford, 1995.
\bibitem{kkm:2000}
M. Kilp, U. Knauer, A. Mikhalev, {\em Monoids, Acts, and Categories}.  Walter de Gruyter, 2000.
\bibitem{McAlister:1976}
D.B. McAlister.  One-to-one partial right translations of a right cancellative semigroup.  {\em J. Algebra}, 43:231-251, 1976.
\bibitem{Miller:2020}
C. Miller.  Semigroups for which every right congruence of finite index is finitely generated.  {\em Monat. Math.}, 193:105-128, 2020.
\bibitem{Pastijn:1975}
F. Pastijn.  A representation of a semigroup by a semigroup of matrices over a group with zero.  \textit{Semigroup Forum}, 10:238-249, 1975.
\bibitem{Robertson:2001} 
E.F. Robertson, N. Ru\v{s}kuc and M.R. Thomson.  On diagonal acts over monoids.  {\em Bull. Australian Math. Soc.}, 63:167-175, 2001.
\bibitem{Robertson:2002} 
E.F. Robertson, N. Ru\v{s}kuc and M.R. Thomson.  On finite generation and other finiteness conditions of wreath products of semigroups.  {\em Comm. Algebra}, 30:3851-3873, 2002.
\bibitem{Thompson:2001} 
M.R. Thomson.  {\em Finiteness Conditions of Wreath Products of Semigroups and Related Properties of Diagonal Acts}.  Ph.D. thesis, University of St. Andrews, 2001.
\bibitem{White:2017}
J.T. White.  Finitely-generated left ideals in Banach algebras on groups and semigroups.  \text{Studia Math.}, 239:67-99, 2017.
\end{thebibliography}
\end{document}